\documentclass[12pt,twoside]{amsart}
 \usepackage{graphicx}
 \newtheorem{theorem}{\bf Theorem}
 \newtheorem{definition}{\bf Definition}
 \newtheorem{note}{\bf Note}
 \newtheorem{result}{\bf Result}

 \newtheorem{remark}{\bf Remark}
 \newtheorem{example}{\bf Example}

\numberwithin{definition}{section} \numberwithin{theorem}{section}
\numberwithin{note}{section} \numberwithin{result}{section}
\numberwithin{corollary}{section} \numberwithin{lemma}{section}
\numberwithin{proposition}{section} \numberwithin{property}{section}
\numberwithin{remark}{section} \numberwithin{example}{section}
\numberwithin{question}{section} \numberwithin{exercise}{section}
\begin{document}
\title{Intuitionistic
Fuzzy Continuity and Uniform Convergence}
\author{Bivas Dinda,\; T. K. Samanta }

\maketitle{Department of Mathematics,Mahishamuri Ramkrishna
Vidyapith,West Bengal,India\\ Department of Mathematics, Uluberia
College,West Bengal, India.
E-mail : bvsdinda@gmail.com,\; mumpu$_{-}$tapas5@yahoo.co.in }
\bigskip
\begin{abstract}A few of the algebraic and topological properties of intutionistic fuzzy continuity and
uniformly intutionistic fuzzy
 continuity are investigated. Also, the concept of uniformly intutionistic fuzzy
 convergence is introduced thereafter a few results on uniformly intutionistic fuzzy convergence are studied.
\end{abstract}
\textbf{Key Words : \;}  Intuitionistic fuzzy norm linear space,
Intuitionistic fuzzy continuity, Cauchy sequence, Uniformly Intuitionistic fuzzy continuity, Uniformly intuitionistic fuzzy convergent.\\
\textbf{2000 Mathematics Subject Classification:} 03F55, 46S40.
\\\\
\section{\textbf{Introduction  \;}} The concept of intuitionistic
fuzzy set, as a generalisation of fuzzy sets \cite{zadeh} was introduced
 by Atanassov \cite{Atanassov}. Intuitionistic fuzzy set is used in
 the process of decision making. Cheng and Moderson \cite{Shih-chuan} introduced the idea of fuzzy norm on a linear space. Bag and Samanta \cite{Bag1} deduce the definition of fuzzy norm whose associated matric is same as the associated metric of Cheng and Moderson \cite{Shih-chuan}.  \\ \\ In this paper after an introduction of intuitionistic fuzzy norm \cite{Samanta} and intuitionistic fuzzy continuity \cite{Samanta} deduced from Bag and Samanta \cite{Bag1} and \cite{Bag2}, it has been shown that the class of intuitionistic fuzzy continuous functions is closed with respect addition,  multiplication, scalar multiplication and inverse operation of multiplication.
 Also, the intuitionistic fuzzy continuity is being characterized
 by open set and a few properties of open sets are also proved in intutionistic fuzzy normed linear space. Thereafter the concept of uniformly intuitionistic
fuzzy continuity is introduced and it is proved that the uniformly intuitionistic
fuzzy continuity implies the intuitionistic fuzzy continuity but not the converse.
\\In the last section, the concept of intuitionistic fuzzy convergence and uniformly intutionistic fuzzy convergence of a sequence of functions are introduced in intutionistic fuzzy normed linear space  and then it is proved that the intuitionistic fuzzy continuity of each term of a sequence of function is transmitted to the limit function under uniformly intutionistic fuzzy convergence of the sequence of functions.
\bigskip
%%%%%%%%%%%%%%%%% Main Body %%%%%%%%%%%%%%%%%%%%%%%%%%%%%%%%%%%%%%%%%%%%%%%%
\bigskip
\section{\textbf{Preliminaries \;} }
 We quote some definitions and statements of a few theorems
which will be needed in the sequel.
\bigskip
\begin{definition} \cite{Schweizer}.
A binary operation \, $\ast \; : \; [\,0 \; , \;
1\,] \; \times \; [\,0 \; , \; 1\,] \;\, \longrightarrow \;\, [\,0
\; , \; 1\,]$ \, is continuous \, $t$ - norm if \,$\ast$\, satisfies
the
following conditions \, $:$ \\
$(\,i\,)$ \hspace{0.5cm} $\ast$ \, is commutative and associative ,
\\ $(\,ii\,)$ \hspace{0.4cm} $\ast$ \, is continuous , \\
$(\,iii\,)$ \hspace{0.2cm} $a \;\ast\;1 \;\,=\;\, a \hspace{1.2cm}
\forall \;\; a \;\; \varepsilon \;\; [\,0 \;,\; 1\,]$ , \\
$(\,iv\,)$ \hspace{0.2cm} $a \;\ast\; b \;\, \leq \;\, c \;\ast\; d$
\, whenever \, $a \;\leq\; c$  ,  $b \;\leq\; d$  and  $a \,
, \, b \, , \, c \, , \, d \;\, \varepsilon \;\;[\,0 \;,\; 1\,]$.
\end{definition}
\smallskip
\begin{definition}
\cite{Schweizer}. A binary operation \, $\diamond \; : \; [\,0 \; ,
\; 1\,] \; \times \; [\,0 \; , \; 1\,] \;\, \longrightarrow \;\,
[\,0 \; , \; 1\,]$ \, is continuous \, $t$-conorm if
\,$\diamond$\, satisfies the
following conditions \, $:$ \\
$(\,i\,)\;\;$ \hspace{0.1cm} $\diamond$ \, is commutative and
associative ,
\\ $(\,ii\,)\;$ \hspace{0.1cm} $\diamond$ \, is continuous , \\
$(\,iii\,)$ \hspace{0.1cm} $a \;\diamond\;0 \;\,=\;\, a
\hspace{1.2cm}
\forall \;\; a \;\; \in\;\; [\,0 \;,\; 1\,]$ , \\
$(\,iv\,)$ \hspace{0.1cm} $a \;\diamond\; b \;\, \leq \;\, c
\;\diamond\; d$ \, whenever \, $a \;\leq\; c$  ,  $b \;\leq\; d$
 and  $a \, , \, b \, , \, c \, , \, d \;\; \in\;\;[\,0
\;,\; 1\,]$.
\end{definition}
\medskip
\begin{remark}
\cite{Vijayabalaji}. $(\,a\,)$ \; For any \, $r_{\,1} \; , \; r_{\,2}
\;\; \in\;\; (\,0 \;,\; 1\,)$ \, with \, $r_{\,1} \;>\;
r_{\,2}$ , there exist $r_{\,3} \; , \; r_{\,4} \;\; \in
\;\; (\,0 \;,\; 1\,)$ \, such that \, $r_{\,1} \;\ast\; r_{\;3}
\;>\; r_{\,2}$ \, and \, $r_{\,1} \;>\; r_{\,4} \;\diamond\;
r_{\,2}$ .
\\\\ $(\,b\,)$ \; For any \, $r_{\,5} \;\,
\in\;\, (\,0 \;,\; 1\,)$ , there exist \, $r_{\,6} \; , \;
r_{\,7} \;\, \in\;\, (\,0 \;,\; 1\,)$ \, such that \,
$r_{\,6} \;\ast\; r_{\,6} \;\geq\; r_{\,5}$ \,and\, $r_{\,7}
\;\diamond\; r_{\,7} \;\leq\; r_{\,5}.$
\end{remark}
\smallskip
\begin{definition}
\cite{Samanta}. Let \,$\ast$\, be a continuous \,$t$-norm ,
\,$\diamond$\, be a continuous \,$t$ - conorm  and \,$V$\, be a
linear space over the field \,$F \;(\, = \; \mathbb{R} \;\, or \;\,
\mathbb{C} \;)$. An \textbf{intuitionistic fuzzy norm} on \,$V$\,
is an object of the form \, $A \;\,=\;\, \{\; (\,(\,x \;,\; t\,)
\;,\; \mu\,(\,x \;,\; t\,) \;,\; \nu\,(\,x \;,\; t\,) \;) \;\, :
\;\, (\,x \;,\; t\,) \;\,\in\;\, V \;\times\;
\mathbb{R^{\,+}} \;\}$ , where $\mu \,,\, \nu\;are\; fuzzy\; sets
\;on \,$V$  \;\times\; \mathbb{R^{\,+}}$ , \,$\mu$\, denotes the
degree of membership and \,$\nu$\, denotes the degree of non -
membership \,$(\,x \;,\; t\,) \;\,\in\;\, V \;\times\;
\mathbb{R^{\,+}}$\, satisfying the following conditions $:$ \\\\
$(\,i\,)$ \hspace{0.28cm}  $\mu\,(\,x \;,\; t\,) \;+\; \nu\,(\,x
\;,\; t\,) \;\,\leq\;\, 1 \hspace{1.2cm} \forall \;\; (\,x \;,\;
t\,)
\;\,\in\;\, V \;\times\; \mathbb{R^{\,+}}\, ;$ \\
$(\,ii\,)$ \hspace{0.28cm}$\mu\,(\,x \;,\; t\,) \;\,>\;\, 0 \, ;$ \\
$(\,iii\,)$ \hspace{0.09cm} $\mu\,(\,x \;,\; t\,) \;\,=\;\, 1$ \; if
and only if \, $x \;=\; \theta \, ;$ \\
$(\,iv\,)$\hspace{0.28cm} $\mu\,(\,c\,x \;,\; t\,) \;\,=\;\,
\mu\,(\,x \;,\; \frac{t}{|\,c\,|}\,)$ \; \; $\;\forall\; c
\;\,\in\;\, F \, $ and $c \;\neq\; 0 \;;$ \\ $(\,v\,)$ \hspace{0.28cm} $\mu\,(\,x
\;,\; s\,) \;\ast\; \mu\,(\,y \;,\; t\,) \;\,\leq\;\, \mu\,(\,x
\;+\; y \;,\; s \;+\; t\,) \, ;$ \\ $(\,vi\,)$ \hspace{0.1cm}
$\mu\,(\,x \;,\; \cdot\,)$ is non-decreasing function of \,
$\mathbb{R^{\,+}}$ \,and\, $\mathop {\lim }\limits_{t\;\, \to
\,\;\infty } \;\,\,\mu\,\left( {\;x\;,\;t\,} \right)=1 ;$
\\ $(\,vii\,)$ \hspace{0.37cm}$\nu\,(\,x \;,\; t\,) \;\,<\;\, 1 \, ;$ \\
$(\,viii\,)$ \hspace{0.1cm} $\nu\,(\,x \;,\; t\,) \;\,=\;\, 0$ \; if
and only if \, $x \;=\; \theta \, ;$ \\ $(\,ix\,)$
\hspace{0.28cm} $\nu\,(\,c\,x \;,\; t\,) \;\,=\;\, \nu\,(\,x \;,\;
\frac{t}{|\,c\,|}\,)$ \; \; $\;\forall\; c
\;\,\in\;\, F \, $ and $c \;\neq\; 0 \;;$ \\ $(\,x\,)$ \hspace{0.28cm} $\nu\,(\,x
\;,\; s\,) \;\diamond\; \nu\,(\,y \;,\; t\,) \;\,\geq\;\, \nu\,(\,x
\;+\; y \;,\; s \;+\; t\,) \, ;$ \\ $(\,xi\,)$ \hspace{0.1cm}
$\nu\,(\,x \;,\; \cdot\,)$ is non-increasing function of \,
$\mathbb{R^{\,+}}$ \,and\, $\mathop {\lim }\limits_{t\;\, \to
\,\;\infty } \;\,\,\nu\,\left( {\;x\;,\;t\,} \right)=0.$
\end{definition}
\smallskip
\begin{definition}
\cite{Samanta}. If $A$ is an intuitionistic fuzzy norm on a linear
space $V$ then $(V\;,\;A)$ is called an intuitionistic fuzzy normed
linear space.
\end{definition}
\smallskip
 For the intuitionistic fuzzy normed linear space \,$(\,V \;,\; A\,)$\,,
 we further assume that \;$\mu,\, \nu,\, \ast,\, \diamond$\,
 satisfy the following axioms : \\\\
$(\,xii\,)$\;\;\;\;\; $\left. {{}_{a\;\; \ast \;\;a\;\; =
\;\;a}^{a\;\; \diamond \;\;a\;\; = \;\;a} \;\;}
\right\}\;\;\;$\,\;\;\;\;\,,\;\;\;forall
$\;\;a\;\; \varepsilon \;\;[\,0\;\,,\;\,1\,].$ \\
$(\,xiii\,)$ \;\; $\mu\,(\,x \;,\; t\,) \;>\; 0 \;\;\;\;,$\; for all$
\;\; t \;>\; 0 \;\; \Rightarrow \;\; x \;=\;\theta\;.$ \\
$(\,xiv\,)$\;\;\, $\nu\,(\,x \;,\; t\,) \;<\; 1 \;\;\;\;\;\; ,$ \;for all$
\;\; t \;>\; 0 \;\; \Rightarrow \;\; x \;=\; \theta\;.$ \\

\smallskip
\begin{definition}
\cite{Samanta}. A sequence $\{x_n\}_n$ in an intuitionistic fuzzy normed linear space $(V\,,\,A)$ is said to \textbf{converge} to $x\;\in\;V$ if for given $r>0,\;t>0,\;0<r<1$, there exist an integer $n_0\;\in\;\mathbb{N}$ such that \\
$\;\mu\,(\,x_n\,-\,x\,,\,t\,)\;>\;1\,-\,r$
 \;\;and\;\; $\nu\,(\,x_n\,-\,x\,,\,t\,)\;<\;r$ \;\;for all $n\;\geq \;n_0$.
\end{definition}
\smallskip
\begin{definition}
\cite{Samanta}. A sequence $\{x_n\}_n$ in an intuitionistic fuzzy normed linear space $(V\,,\,A)$ is said to be \textbf{cauchy sequence} if $\mathop {\lim }\limits_{n\;\,
\to \,\;\infty } \;\,\,\mu(x_{n+p}-x_n ,t)=1\; $ and $\mathop {\lim }\limits_{n\;\,
\to \,\;\infty } \;\,\,\nu(x_{n+p}-x_n ,t)=0\;\;,\;p=1,2,3,..... $
\end{definition}
\smallskip
\begin{definition}
\cite{Samanta}. Let, $(\;U\;,\;A\;)$ and $(\;V\;,\;B\;)$ be two
intuitionistic fuzzy normed linear space over the same field $F$. A
mapping $f$ from $(\;U\;,\;A\;)$ to $(\;V\;,\;B\;)$ is said to be \textbf{
intuitionistic fuzzy continuous} at $x_0\;\in\;U$, if for any given
$\epsilon\;>\;0\;,\alpha\;\in\;(0,1)\;,\exists\;\delta
\;=\delta(\alpha,\epsilon)\;>0\;,\beta\;=\beta(\alpha,\epsilon)\;\in\;(0,1)$
such that for all $x\;\in\;U$,
\[\mu_U(x-x_0 \;,\;\delta)\;>\;1-\beta\;\Rightarrow\;
\mu_V(f(x)-f(x_0) \;,\;\epsilon)\;>\;1-\alpha\] \[
\nu_U(x-x_0 \;,\;\delta)\;<\;\beta\;\Rightarrow\;
\nu_V(f(x)-f(x_0) \;,\;\epsilon)\;<\;\alpha \]
\end{definition}
\smallskip
\begin{definition}
\cite{Samanta}. A mapping $f$ from $(U\,,\,A)$ to $(V\,,\,B)$ is said to be \textbf{sequentially intuitionistic fuzzy continuous} at $x_0\;\in\;U$, if for any sequence $\{x_n\}_n$, $x_n\;\in\;U\;,\;\forall\;n\;\in\;\mathbb{N}$ with $x_n\;\rightarrow\;x_0$ in $(U\,,\,A)$ implies
$f(x_n)\;\rightarrow\;f(x_0)$ in $(V\,,\,B)$, that is
\[\mathop {\lim }\limits_{n\;\,
\to \,\;\infty } \;\mu_U(x_n-x_0 \,,\,t)\;=\;1 \;and \; \mathop {\lim }\limits_{n\;\,
\to \,\;\infty } \;\nu_U(x_n-x_0 \,,\,t)\;=\;0\; \] \[\Rightarrow\;\mathop {\lim }\limits_{n\;\,
\to \,\;\infty } \;\mu_{\,V}(f(x_n)-f(x_0)\,,\,t)\;=\;1 \; and \;
\mathop {\lim }\limits_{n\;\,
\to \,\;\infty } \;\nu_{\,V}(f(x_n)-f(x_0)\,,\,t)\;=\;0\]
\end{definition}
\smallskip
\begin{theorem}
\cite{Samanta}. Let, $f$ be a mapping from $(U\,,\,A)$ to $(V\,,\,B)$.
Then $f$ is intuitionistic fuzzy continuous on $U$ if and only if
it is sequentially intuitionistic fuzzy continuous on $U$
\end{theorem}
\bigskip
\section{\textbf{Algebra of Intuitionistic Fuzzy Continuous functions.}}
   In this section, consider $(\,U\;,\;A\,)$ and $(\,V\;,\;B\,)$
   be any two intuitionistic fuzzy normed
linear space over the same field $F$.
\\
\begin{theorem}
If $f\;:(\,U\;,\;A\,)\;\rightarrow\;(\,V\;,\;B\,)$ and $g\;:(\,U\;,\;A\,)\;\rightarrow\;(\,V\;,\;B\,)$ are two sequentially
intuitionistic fuzzy continuous functions and $(\,U\;,\;A\,)$
and $(\,V\;,\;B\,)$ satisfies the condition $(xii)$
then $f\,+\,g$ \,,\, $k\;f$, where $k\,\in\,F$ are
also sequentially intuitionistic fuzzy continuous functions
over the same field $F$.
\end{theorem}
\begin{proof}
Let, \,$\{{x_n}\}_n$ be a sequence in $U$ such that
$x_n\;\rightarrow\;x$ \,in \,$(\,U\,,\,A\,)$.
Thus $\forall\;t\;\in\;\mathbb{R}$ we have\[\mathop
{\lim }\limits_{n\; \to \;\infty } \;\mu_{\,U} (\,x_{\,n}\,-\,x\,
,\,t\,)\;=\;1\;\;\;\; and\;\;\; \mathop {\lim }\limits_{n\;
\to \;\infty } \;\nu_{\,U} (\,x_{\,n}\,-\,x \,,\,t\,)\;=\;0
\hspace{0.2cm} \cdots
\hspace{0.3cm}(1)\] Since $f$
and $g$ are sequentially intuitionistic fuzzy continuous at $x$ from
(1), we have \[\mathop {\lim }\limits_{n\; \to
\;\infty } \;\mu_{\,V} (f(x_n)-f(x) \,,\, t)\;=\;1\;,\;\mathop
{\lim }\limits_{n\; \to \;\infty } \;\nu_{\,V} (f(x_n)-f(x)
\,,\, t)\;=\;0\;,\;\,\forall\;t\;\in\;\mathbb{R}\]
and \[\mathop {\lim }\limits_{n\; \to \;\infty }
\;\mu_{\,V} (g(x_n)-g(x) \,,\, t)\;=\;1\;\,,\;\,\mathop {\lim
}\limits_{n\; \to \;\infty }\;\nu_{\,V} (g(x_n)-g(x) \,,\, t)\;=\;0\;,\;\;\forall\;t\;\in\;\mathbb{R}\]
\\Now, $\;\;\;\mu_{\,V} (\,(f\,+\,g)(x_n)\;-\;(f\,+\,g)(x)
\,,\, t\,)$ \[=\;\mu_{\,V}(\,f(x_n)\;-\;f(x)\;+\;g(x_n)\;-\;g(x)\,,\,t\,)\;\]
\[\hspace{2.1cm}\geq\;\mu_{\,V}\left(\,f(x_n)\;-\;f(x)\,,\,
\frac{t}{2}\,\right)\;\ast\;
\mu_{\,V}\left(\,g(x_n)\;-\;g(x)\,,\, \frac{t}{2}\,\right)\;\] \\
 Taking limit we have,\\ $ \mathop
{\lim }\limits_{n\;\to\;\infty}\;\mu_{\,V} (\,(f\,+\,g)(x_n)\;-\;(f\,+\,g)(x)
\,,\, t\,)$ \[\hspace{1.5cm}\geq\;\mathop
{\lim }\limits_{n\;\to \;\infty}\mu_{\,V}\left(\,f(x_n)\,-\, f(x)\,,\,
\frac{t}{2}\,\right)\;\ast\;\mathop
{\lim }\limits_{n\;\to \;\infty}\mu_{\,V}
\left(\,g(x_n)\,-\,g(x)\,,\, \frac{t}{2}\,\right)\;=\;1\;\ast\;1\;=\;1.\]
\\Again,$\;\;\;\nu_{\,V} (\,(f\,+\,g)(x_n)\;-\;(f\,+\,g)(x)
\,,\, t\,)$ \[=\;\nu_{\,V}(\,f(x_n)\;-\;f(x)\;+\;g(x_n)\;-\;g(x)\,,\,t\,)\;\]
\[\hspace{2.1cm}\leq\;\nu_{\,V}\left(\,f(x_n)\;-\;f(x)\,,\, \frac{t}{2}\,\right)\;\diamond\;
\nu_{\,V}\left(\,g(x_n)\;-\;g(x)\,,\, \frac{t}{2}\,\right)\;\]
\\ Taking limit we have,\\$ \mathop
{\lim }\limits_{n\;\to\;\infty}\;\nu_{\,V} (\,(f\,+\,g)(x_n)\;-\;(f\,+\,g)(x)
\,,\, t\,)$ \[\hspace{1.5cm}\leq\;\mathop
{\lim }\limits_{n\;\to \;\infty}\nu_{\,V}\left(\,f(x_n)\,-\, f(x)\,,\,
\frac{t}{2}\,\right)\;\diamond\;\mathop
{\lim }\limits_{n\;\to \;\infty}\nu_{\,V}
\left(\,g(x_n)\,-\,g(x)\,,\, \frac{t}{2}\,\right)\;=\;0\;\diamond\;0\;=\;0.\]
\\So, $f\,+\,g$ is sequentially intuitionistic fuzzy
continuous.\\\\Obviously, $k\,f$ is sequentially intuitionistic fuzzy
continuous for every $k\in\;F$.
\end{proof}
\bigskip
\smallskip
 \textbf{We further assume that,} for an intuitionistic fuzzy normed linear space $(\,V\;,\;A\,)$ \,and for\, $x\;\neq\;\theta$,\\
 $(xv)\;\;\;\mu(x\,,\, .)$ is a continuous function of $\mathbb{R}$ and strictly increasing on the subset $\{\, t\;\,:\;\,0\;<\;\mu(x\,,\,t)\;<\;1 \, \}$ of $\mathbb{R}$.\\$(xvi)\;\;\;\nu(x\,,\, .)$ is a continuous function
 of $\mathbb{R}$ and strictly decreasing on the
 subset $\{\,t\;\,:\;\,0\;<\;\nu(x\,,\,t)\;<\;1 \,\}$ of $\mathbb{R}$.

 \medskip

 \begin{theorem}
 If $f\;:(\,U\;,\;A\,)\;\rightarrow\;(\,V\;,\;B\,)$ and $g\;:(\,U\;,\;A\,)\;\rightarrow\;(\,V\;,\;B\,)$
 are two sequentially intuitionistic fuzzy continuous functions
 and $(\,U\;,\;A\,)$ and $(\,V\;,\;B\,)$ satisfies $(xii)$\,,\,
  $(xv)$ and $(xvi)$ then \\$(a)\;\;f\,g\;$ is sequentially
  intuitionistic fuzzy continuous functions over the same field $F$,
  \\$(b)$  if $g(x)\;\neq\;0\;,\;\;\forall\;x\;\in\;U$ then \,
  $\frac{f}{g}$ \, is sequentially intuitionistic fuzzy
  continuous functions over the same field $F$.
 \end{theorem}

 \begin{proof}
 \;(a)\;\,Let,$\;\{\,{x_n}\,\}_n\;$ be a sequence in $U$ such that
$x_n\;\rightarrow\;x\;$ in $(\,U\;,\;A\,)$.
Thus \;$\forall\;t\;\in\;\mathbb{R}$ \;we have\\\[\mathop
{\lim }\limits_{n\; \to \;\infty } \;\mu_{\,U}\, (\,x_n\,-\,x\,
,\,t\,)\;=\;1\;\;\;\; and\;\;\; \mathop {\lim }\limits_{n\;
\to \;\infty } \;\nu_{\,U}\,(\,x_n\,-\,x \,,\,t\,)\;=\;0 \hspace{0.2cm} \cdots
\hspace{0.3cm}(2)\]
\\Since $f$ and $g$ are sequentially intuitionistic fuzzy
continuous at $x$ , from (2), we have
\\ \[\mathop {\lim }\limits_{n\; \to
\;\infty } \;\mu_{\,V}\,(\,f(x_n)\,-\,f(x) \,,\, t\,)\;=\;1\;,\;\mathop
{\lim }\limits_{n\; \to \;\infty } \;\nu_{\,V}\,(\,f(x_n)\,-\,f(x)
\,,\, t\,)\;=\;0\;,\;\,\forall\;t\;\in\;\mathbb{R}\]
and \[\mathop {\lim }\limits_{n\; \to \;\infty }
\;\mu_{\,V}\,(\,g(x_n)\,-\,g(x) \,,\, t\,)\;=\;1\;\,,\;\,\mathop {\lim
}\limits_{n\; \to \;\infty }\;\nu_{\,V}\,(\,g(x_n)\,-\,g(x) \,,\, t\,)\;=\;0\;,\;\;\forall\;t\;\in\;\mathbb{R}\]
\\ Now,\;$\;\mu_{\,V}\,(\,(f\,g)(x_n)\;-\;(f\,g)(x_0)\,,\, t\,)$
\[=\;\mu_{\,V}\,(\,f(x_n)\,(g(x_n)\,-\,g(x_0))
\;+\;g(x_0)\,(\,f(x_n)\,-\,f(x_0)\,)\,,\, t\,)\]
\[ \hspace{1.5cm}=\;\mu_{\,V}\,(\,(\,f(x_n)\,-\,f(x_0)\,)\;
(\,g(x_n)\,-\,g(x_0)\,)\;+\;f(x_0)
\;(\,g(x_n)\,-\,g(x_0)\,)\;\] \[ \hspace{4.5cm}+\;g(x_0)\;(\,f(x_n)
\,-\,f(x_0)\,)\;,\; t)
\] \[ \hspace{1.7cm}\geq\;
\mu_{\,V}\,\left(\,(\,f(x_n)\,-\,f(x_0)\,)\;(\,g(x_n)\,-\,g(x_0)\,)
\,,\,\frac{t}{3}\,\right)\;\ast\;
\mu_{\,V}\,\left(\,f(x_0)\;(\,g(x_n)\,-\,g(x_0)\,)\,,\,\frac{t}{3}\,\right)\;\] \[ \hspace{5.1cm}
\ast\;\mu_{\,V}\,\left(\,g(x_0)\;
(\,f(x_n)\,-\,f(x_0)\,)\,,\,\frac{t}{3}\,\right)\] \[\hspace{1.7cm}=\;\mu_{\,V}\,\left(\,f(x_n)\,-\,f(x_0)\,,\,\frac{t}
{3\mid\,g(x_n)\,-\,g(x_0)\,\mid}\,\right) \;\ast\;\mu_{\,V}\,
\left(\,g(x_n)\,-\,g(x_0)\,,\,
\frac{t}{3\mid\,f(x_0)\,\mid}\right)\]
\[\hspace{5.4cm}\ast\;\mu_{\,V}\,\left(\,f(x_n)\,-\,f(x_0)\,,\,
\frac{t}{3\mid\,g(x_0)\,\mid}\,\right)\]
\\Taking limit as $n\;\rightarrow\;\infty$ we have,\\\\
$\mathop {\lim}\limits_{n\;\to \;\infty }\,
\mu_{\,V}\,(\,(f\,g)(x_n)\,-\,(f\,g)(x_0)\,,\,t\,)$\\
\[ \hspace{1.0cm} \geq\;\mathop {\lim}\limits_{n\;\to \;\infty }\,\mu_{\,V}\,\left(\,f(x_n)\,-\,f(x_0)\,,\,
\frac{t}{3\mid\,g(x_n)\,-\,g(x_0)\,\mid}\,\right)\;\ast\;
\mathop {\lim}\limits_{n\;\to \;\infty }\,\mu_{\,V}\,\left
(\,g(x_n)\,-\,g(x_0)\,,\,
\frac{t}{3\mid\,f(x_0)\,\mid}\,\right)\]
\[\hspace{3.5cm} \ast\;\mathop {\lim}\limits_{n\;\to \;\infty }\,\mu_{\,V}\,\left(\,f(x_n)\,-\,f(x_0)\,,\, \frac{t}{3\mid\,g(x_0)\,\mid}\,\right)\]
\\\[\hspace{1.0cm} =\;\mu_{\,V}\,\left(\,f(x_n)\,-\,f(x_0)\,,\,
\mathop {\lim}\limits_{n\;\to \;\infty }\;\frac{t}{3\mid\,g(x_n)\,-\,g(x_0)\,\mid}\,\right)\;\ast\;
\mathop {\lim}\limits_{n\;\to \;\infty }\,\mu_{\,V}\,\left(\,g(x_n)\,-\,g(x_0)\,,\,
\frac{t}{3\mid\,f(x_0)\,\mid}\,\right)\]
\[\hspace{5.2cm} \ast\;\mathop {\lim}\limits_{n\;\to\;\infty }\,\mu_{\,V}\,\left(\,f(x_n)\,-\,f(x_0) \,,\, \frac{t}{3\mid\,g(x_0)\,\mid}\,\right)\;
\;, \;\;\;by(vii)\]
\[=\;\mu_{\,V}\,(\,f(x_n)\,-\,f(x_0),\infty\,)\;\ast\;1\;\ast\;1\\
=\;1\;\ast\;1\;\ast\;1\\=\;1 \hspace{2.4cm}\]
and\\
$\nu_{\,V}(\,(f\,g)(x_n)\;-\;(f\,g)(x_0)\,,\, t\,)$
\[=\;\nu_{\,V}(\,f(x_n)\,(\,g(x_n)\,-\,g(x_0)\,)
\;+\;g(x_0)\,(\,f(x_n)\,-\,f(x_0)\,)\,,\, t\,)\]
\[ \hspace{1.5cm}=\;\nu_{\,V}(\,(\,f(x_n)\,-\,f(x_0)\,)\;(\,g(x_n)\,-\,g(x_0)\,)\;+\;f(x_0)
\,(\,g(x_n)\,-\,g(x_0)\,)\;\] \[ \hspace{4.5cm}+\;g(x_0)\,(\,f(x_n)\,-\,f(x_0)\,)\;,\; t\,)
\] \[ \hspace{1.5cm}\leq\;
\nu_{\,V}\left(\,(\,f(x_n)\,-\,f(x_0)\,)\;(\,g(x_n)\,-\,g(x_0)\,)\,,
\,\frac{t}{3}\,\right)\;\diamond\;
\nu_{\,V}\left(\,f(x_0)\;(\,g(x_n)\,-\,g(x_0)\,)\,,\,\frac{t}{3}\,\right)\;\] \[ \hspace{5.1cm}
\diamond\;\nu_{\,V}\left(\,g(x_0)\;
(\,f(x_n)\,-\,f(x_0)\,)\,,\,\frac{t}{3}\,\right)\] \[\hspace{1.5cm}=\;\nu_{\,V}\left(\,f(x_n)\,-\,f(x_0)\,,\,\frac{t}
{3\mid\,g(x_n)\,-\,g(x_0)\,\mid}\,\right) \;\diamond\;\nu_{\,V}\left(\,g(x_n)\,-\,g(x_0)\,,\,
\frac{t}{3\mid\,f(x_0)\,\mid}\,\right)\]
\[\hspace{5.4cm}\diamond\;\nu_{\,V}\left(\,f(x_n)\,-\,f(x_0)\,,\,
\frac{t}{3\mid\,g(x_0)\,\mid}\,\right)\]\\\\
Taking limit as
$n\;\rightarrow\;\infty$ we have,\\\\
$\mathop {\lim}\limits_{n\;\to \;\infty }\,
\nu_{\,V}(\,(f\,g)(x_n)\,-\,(f\,g)(x_0)\,,\,t\,)$\\
\[ \hspace{1.0cm} \leq\;\mathop {\lim}\limits_{n\;\to \;\infty }\,\nu_{\,V}\left(\,f(x_n)\,-\,f(x_0)\,,\,
\frac{t}{3\mid\,g(x_n)-g(x_0)\,\mid}\,\right)\;\diamond\;
\mathop {\lim}\limits_{n\;\to \;\infty }\,\nu_{\,V}\left(\,g(x_n)-g(x_0)\,,\,
\frac{t}{3\mid\,f(x_0)\,\mid}\,\right)\]
\[\hspace{3.5cm} \diamond\;\mathop {\lim}\limits_{n\;\to \;\infty }\,\nu_{\,V}\left(\,f(x_n)\,-\,f(x_0)\,,\, \frac{t}{3\mid\,g(x_0)\,\mid}\,\right)\]
\\\[\hspace{1.0cm} =\;\nu_{\,V}\left(\,f(x_n)\,-\,f(x_0)\,,\,
\mathop {\lim}\limits_{n\;\to \;\infty }\;\frac{t}{3\mid\,g(x_n)-g(x_0)\,\mid}\,\right)
\;\diamond\;
\mathop {\lim}\limits_{n\;\to \;\infty }\,\nu_{\,V}\left(g(x_n)\,-\,g(x_0)\,,\,
\frac{t}{3\mid\,f(x_0)\,\mid}\,\right)\]
\[\hspace{5.2cm} \diamond\;\mathop {\lim}\limits_{n\;\to\;\infty }\,\nu_{\,V}\left(\,f(x_n)\,-\,f(x_0) \,,\, \frac{t}{3\mid\,g(x_0)\,\mid}\,\right)\;
\;, \;\;\;by(vii)\]
\[=\;\nu_{\,V}(\,f(x_n)\,-\,f(x_0)\,,\,\infty\,)\;\diamond\;0\;\diamond\;0\\
=\;0\;\diamond\;0\;\diamond\;0\\=\;0 \hspace{2.4cm}\]
Hence the proof.\\\\
\;(b)\;We now show that $\frac{1}{g}$ is
sequentially intuitionistic fuzzy continuous at $x$ if $g(x)\neq\;0$  for all $x\;\in\;U$.\\
\[\mu_{\,V}\left(\,\frac{1}{g}(x_n)\,-\,\frac{1}{g}(x_0)
\,,\, t\,\right)\;=\;\mu_{\,V}\left(\,\frac{g(x_n)\,-\,g(x_0)}
{g(x_n)\,g(x_0)}\;,\;t\,\right)\;
=\;\mu_{\,V}\left(\,\frac{1}{g(x_n)\,g(x_0)}\;,\;\frac{t}
{g(x_n)\,-\,g(x_0)}\,\right)\] \\
Taking limit as $n\;\rightarrow\;\infty$ we have,\\\\
$\mathop {\lim}\limits_{n\;\to\;\infty }
\mu_{\,V}\left(\,\frac{1}{g}(x_n)\;-\;\frac{1}{g}(x_0)\;,\;t\,\right)$\\
\[=\;\;\mu_{\,V}\left(\,\frac{1}{g(x_n)\,g(x_0)}\;\;,\;
\mathop {\lim}\limits_{n\;\to \;\infty }\frac{t}{g(x_n)\;-\;g(x_0)}\,\right)
 \;\;\;by (vii)\]\\
\[=\;\;\mu_{\,V}\left(\,\frac{1}{g(x_n)\,g(x_0)}\;,\;\infty\,\right)\;\;=\;\;1.
 \hspace{4.0cm}\] \\
Again ,\\
\[\nu_{\,V}\left(\,\frac{1}{g}(x_n)\,-\,\frac{1}{g}(x_0)
\,,\, t\,\right)\;=\;\nu_{\,V}\left(\,\frac{g(x_n)\,-\,g(x_0)}
{g(x_n)\,g(x_0)}\;,\;t\,\right)\;
=\;\nu_{\,V}\left(\,\frac{1}{g(x_n)\,g(x_0)}\;,\;\frac{t}{g(x_n)\,-\,g(x_0)}\,\right)\]
\\\\
Taking limit as $n\;\rightarrow\;\infty$ we have,\\\\
$\mathop {\lim}\limits_{n\;\to\;\infty }
\nu_{\,V}\left(\,\frac{1}{g}(x_n)\;-\;\frac{1}{g}(x_0)\;,\;t\,\right)$\\
\[=\;\;\nu_{\,V}\left(\,\frac{1}{g(x_n)\,g(x_0)}\;\;,\;
\mathop {\lim}\limits_{n\;\to \;\infty }\frac{t}{g(x_n)\;-\;g(x_0)}\,\right)
 \;\;\;by (vii)\]
\[=\;\;\nu_{\,V}\left(\,\frac{1}{g(x_n)\,g(x_0)}\;,\;\infty\,\right)\;\;=\;\;0.
 \hspace{3.5cm}\] \\
Hence $\frac{1}{g}$ is sequentially intuitionistic fuzzy continuous.\\
The proof is complited by considering the product of $f$ and $\frac{1}{g}$.
 \end{proof}
 \medskip
\begin{note}
Let, $(\,V=\;\mathbb{R}\;,\;\parallel\cdot\parallel\,)$ be a normed
linear space and define \,$a \;\ast\; b \;\,=\;\, \min\,\{\;a \;,\;
b\;\}$ \,and\, $a \;\diamond\; b \;\,=\;\, \max\,\{\;a \;,\; b\;\}$
\; for all \; $a \;,\; b \;\,\in\;\, [\,0 \;,\; 1\,]$ . For
all \,$t \;>\; 0$. Define , \; $\mu\,(\,x \;,\; t\,) \;\,=\;\,
\frac{t}{t \;+\; k\;\|\,x\,\|}$ \, and \, $\nu\,(\,x \;,\; t\,)
\;\,=\;\, \frac{k\;\|\,x\,\|}{t \;+\; k\;\|\,x\,\|}$ \, where \, $k
\;>\; 0$.\; It is easy to see that \, $A \;\,=\;\, \{\; (\,(\,x \;,
\; t\,) \;,\; \mu\,(\,x \;,\; t\,) \;,\; \nu\,(\,x \;,\; t\,) \;) \;\, :
\;\, (\,x \;,\; t\,) \;\,\in\;\, V \;\times\;
\mathbb{R^{\,+}} \;\}$ is an intuitionistic fuzzy norm on \,$V$. Let
$f\;:\;\mathbb{R}\;\rightarrow\;\mathbb{R}$.Then  $f$ is
continuous on $(\,V \;,\;\parallel\cdot\parallel\,)$ if and only if it is
intuitionistic fuzzy continuous on $(\,V\;,\;A\,)$.
\end{note}
\smallskip
\begin{proof}
By example (2) of \cite{Samanta}, $\{x_n\}_n$ is convergent in \,
$(\,V \;,\;\parallel\cdot\parallel\,)$ \,if and only if $\{x_n\}_n$ is
convergent in \, $(\,V\;,\;A\,)$. So, $f$ is continuous on\,
$(\,V \;,\;\parallel\cdot\parallel\,)$ \\\\
$\Leftrightarrow\;$For any sequence \,$\{x_n\}_n$ \,converging to $x$ in \,
$(\,V \;,\;\parallel\cdot\parallel\,)$, \, $\{f(x_n)\}_n$ \, converges to
$f(x)$ \,in \,$(\,V \;,\;\parallel\cdot\parallel\,)$.\\\\
$\Leftrightarrow\;$For any sequence $\{x_n\}_n$ converging to $x$ in \,
$(\,V\;,\;A\,)$, \, $\{f(x_n)\}_n$ converges to $f(x)$ in \,
$(\,V\;,\;A\,)$.\\$\Leftrightarrow\;\;f$ is continuous on \,
$(\,V\;,\;A\,)$.
\end{proof}
\smallskip
\begin{definition}
Let, $0\;<\;r\;<\;1\;,\;t\;\in\;\mathbb{R}^+$ and $\;x\;\in\;V$.
Then
the set \\
         $\;B(\,x\,,\,r\,,\,t\,)\;=\;\{\;y\;\in\;V\;:\;\mu(\,x\,-\,y \,,\, t\,)
         \;>\;1-r\;,\;\;\nu(\,x-y \,,\, t\,)\;<\;r\;\}$
is called an \textbf{open ball} in \,$(\,V\;,\;A\,)$ \,with $x$ as its
center and $r$ as its radious with respect to $t$.
\end{definition}
\smallskip
\begin{definition}
A subset \,$G$ \,of\, $V$ \,is said to be an \textbf{open set} in
$(\,V\;,\;A\,)$ \,if for each $x\;\in\;G$ there exist \,
$r_x\;\in\;(\,0\;,\;1\,)$ \,and\, $t\;\in\;\mathbb{R}^+$ \,such that
$B(\,x\,,\,r_x\,,\,t\,)\;\subseteq\;G$.
\end{definition}
\smallskip
\begin{theorem}
Every open ball \,$B(\,x\,,\,r\,,\,t\,)$ \,in\, $(\,V\,,\,A\,)$
is an open set in \,$(\,V\,,\,A\,)$
\end{theorem}
\begin{proof}
Let, $B(\,x\,,\,r\,,\,t\,)$ \,be an open ball with center at $x$ and
radious $r$ with respect to $t$. Then,\\
\[\mu(\,x\,-\,y \,,\,t\,)\;>\;1-r\,\;\; and \;\;\,
\nu(\,x\,-\,y \,,\,t\,)\;<\;r \hspace{2.5cm}(3).\]
\\Then for every $t_0\;\in\;(\,0\,,\,t\,)$, the relation (3)
 is true. So, for $t_0\;\in\;(\,0\,,\,t\,)$, \\
  \[\mu(\,x\,-\,y \,,\,t_{0}\,)\;
 >\;1-r\,\;\; and \;\;\,
\nu(\,x\,-\,y \,,\,t_{0}\,)\;<\;r \]
\\ Let, $r_0\;=\;\mu(\,x\,-\,y \,,\,t_0\,)$. Since,
$r_0\;>\;1-r\;,\;\exists\;s\;\in\;(\,0\,,\,1\,)$ such that
$r_0\;>\;1-s\;>\;1-r$.\\ Now for given $r_0$ and $s$ such that $r_0\;>\;1-s\;,\;\exists\;r_1 \,,\, r_2\;\in\;(\,0\,,\,1\,)$
such that $r_0\;\ast\;r_1\;>\;1-s$ \,and\, $(\,1\,-\,r_0\,)
\;\diamond\;(\,1\,-\,r_2\,)\;<\;s$\\
Let, $r_3\;=\;max\;\{\,r_1 \,,\,r_2\,\}$.\\Then,
$r_0\;\ast\;r_1\;\leq\;r_0\;\ast\;r_3$\; and\; $r_2\;\leq\;r_3\;\\\Rightarrow\;1\,-\,r_3\;\leq\;1\,-\,r_2\;\\
\Rightarrow \;(\,1\,-\,r_0\,)\;\diamond\;(\,1\,-\,r_3\,)
\;\leq\;(\,1\,-\,r_0\,)\;\diamond\;(\,1\,-\,r_2\,)$.\\These implies that,\\$1\,-\,s\;<\;r_0\;\ast\;r_1\;\leq\;r_0\;\ast\;r_3\;$ and $\;(\,1\,-\,r_0\,)\;\diamond\;(\,1\,-\,r_3\,)
\;\leq\;(\,1\,-\,r_0\,)\;\diamond\;(\,1\,-\,r_2\,)
\;<\;s\;$ \\ i.e.,\;$r_0\;\ast\;r_3\;>\;1\,-\,s\;$ and $\;(\,1\,-\,r_0\,)\;\diamond\;(\,1\,-\,r_3\,)\;<\;s.$
\\Consider the open ball \,$B\,(\,y\,,\,1\,-\,r_3\,,\,t\,-\,t_0\,).$
 \\It is sufficient to show that $B\,(\,y\,,\,1\,-\,r_3\,,\,t\,-\,t_0\,)\;\subset\;B\,(\,x\,,\,r\,,\,t\,)$.
 \\ Let, $z\;\in\;B\,(\,y\,,\,1\,-\,r_3\,,\,t\,-\,t_0\,).$
 \\Then\;\;$\mu(\,y\,-\,z\;,
\;t\,-\,t_0\,)\;>\;r_3$ \;and\; $\nu(\,y-z\;,\;t\,-\,t_0\,)\;<\;1\,-\,r_3$.
\;Therefore,\[\;\mu\;(\,x\,-\,z \,,\,t\,)
\;\;=\;\;\mu\;(\,x\,-\,y\,+\,y\,-\,z \,,\,t_0\;+\;(\,t\,-\,t_0\,))\]
\[\hspace{2.5cm}\geq\;\mu\,(\,x\,-\,y \,,\,t_0\,)\;
\ast\;\mu\,(\,y\,-\,z \,,\,t\,-\,t_0\,)\]
\[\hspace{1.5cm}>\;r_0\;\ast
\;r_3\;\;>\;\;1-s\;\;>\;\;1\,-\,r.\]\\and
\[\nu\;(\,x\,-\,z \,,\,t\,)\;\;=\;\;
\nu\;(\,x\,-\,y\,+\,y\,-\,z \,,\,t_0\;+\;(\,t\,-\,t_0\,))\]
\[\hspace{2.6cm}\leq\;\nu\;(\,x\,-\,y \,,\,t_0\,)\;\diamond\;\nu\;(\,y\,-\,z \,,\,t\,-\,t_0\,)\;\]\[\hspace{2.3cm}<\;(\,1\,-\,r_0\,)\;\diamond
\;(\,1\,-\,r_3\,)\;\;<\;\;s\;<\;\;r.\]\\Thus\;
$z\;\in\;B\,(\,x\,,\,r\,,\,t\,)$ \,and hence $\;B\,(\,y\,,\,1\,-\,r_3\,,\,t\,-\,t_0\,)\;\subset\;B\,(\,x\,,\,r\,,\,t\,).$
\end{proof}
\medskip
\begin{definition}
A subset \,$N$ \,of \,$V$ \,is said to be a $\textbf{neighbourhood}$ of\,
$x\;(\,\in\;V\,)$ \,in \,$(\,V\,,\,A\,)$ \,if there exist $r\;\in\;(\,0\,,\,1\,)$
and \,$t\;\in\;\mathbb{R}^+$ \,such that \,
$B\,(\,x\,,\,r\,,\,t\,)\;\subset\;N$.
\end{definition}
\smallskip
\begin{theorem}
The following statements are equivalent:\\
$(i)\;\;\;\;f$ is intuitionistic fuzzy continuous on $U$.\\
$(ii)\;\;\;P$ is open in
$(\,V\;,\;B\,)\;\Rightarrow\;f^{-1}\;(P)$ is open in
$(\,U\;,\;A\,)$.\\$(iii)\;$ For each $x\;\in\;U\;,\;N$
is a neighbourhood of \, $f(x)$ \, in\,
$(\,V\;,\;B\,)\;\Rightarrow\;f^{-1}\;(N)$ \,is a neighbourhood of \,$x$\, in
$(\,U\;,\;A\,)$.
\end{theorem}
\begin{proof}
$(i)\;\Rightarrow\;(ii)\;:\;$ Suppose $f$ is intuitionistic fuzzy
continuous on \,$U$ \,and \,$P$ \,is open in $(\,V\;,\;B\,)$. If
$f^{-1}\,(P)\;=\;\phi$, then their is nothing to prove.\\
Let, $f^{-1}\,(P)\;\neq\;\phi$ \,and\, $x_0\;\in\;f^{-1}\,(P)$.
Then $f(x_0)\;\in\;P$. So, there exist \, $\epsilon\,(\;>\;0\,)$ \,
and\, $\alpha\;\in\;(\,0 \,,\, 1\,)$ \,such that\,
 $B\,(\,f(x_0) \,,\, \alpha \,,\, \epsilon\,)\;
\subset\;P$. Since \,$f$\, is intuitionistic fuzzy continuous on \,$U$,
there exist \,$\delta\,(\;>0\,)$ \,and\, $\beta\;\in\;(\,0 \,,\, 1\,)$
such that for all \,$x\;\in\;U$,\\
\[\mu_{\,U}(\,x\,-\,x_0 \,,\, \delta\,)\;>\;1\,-\,\beta\;\Rightarrow
\;\mu_{\,V}(\,f(x)\,-\,f(x_0) \,,\,\epsilon \,)\;>\;1\,-\,\alpha\;\]
\[\nu_{\,U}(\,x\,-\,x_0 \,,\, \delta\,)
\;<\;\beta\;\;\Rightarrow \;\;\nu_{\,V}(\,f(x)\,-\,f(x_0)
\,,\, \epsilon\,)\;<\;\alpha\;\;\;\]
\\i.e.,$\;x\;\in\;B\,(\,x_0 \,,\,\beta \,,\, \delta\,)
\;\Rightarrow\;f(x)\;\in\;B\,(\,f(x_0) \,,\,
\alpha \,,\, \epsilon\,)\;\subset\;P$ \\
$\Rightarrow\;\;B\,(\,x_0 \,,\, \beta \,,\, \delta\,)
\;\subset\;f^{-1}\,(P)$\\
$\Rightarrow\;\;f^{-1}\,(P)$ is open in \,$(\,U \,,\, A\,)$.
\\\\ $(ii)\;\Rightarrow\;(i)\;:\;$ Let, \,
$\epsilon\,(\;>\;0\,)$ \,and\, $\alpha\;\in\;(\,0 \,,\,1\,)$
\,and\, $x_0\;\in\;U$. \,Then \,$B\,(\,f(x_0) \,,\,
\alpha \,,\, \epsilon\,)$ \,is open in \,$(\,V \,,\, B\,)$.\\ $\Rightarrow\;\;f^{-1}\,(\,B\,(\,f(x_0)\,,\,
\alpha \,,\, \epsilon\,)\,)$ \,is open in \,$(\,U \,,\, A\,)$
containing \,$x_0$. \\$\Rightarrow\;\;\exists\;\delta\;>\;0$ \,
and\, $\beta\;\in\;(\,0 \,,\, 1\,)$ \,such that\, $B\,(\,x_0 \,,\, \beta \,,\, \delta\,)\;\subset\;f^{-1}\,(\,B\,(\,f(x_0)\,,\,\alpha\,,\,\epsilon\,)\,).$
\\$\Rightarrow\;f\;(\,B\,(\,x_0 \,,\, \beta \,,\, \delta\,)\,)\;
\subset\;B\,(\,f(x_0) \,,\, \alpha \,,\, \epsilon \,)$.\\
$\;\Rightarrow\;\;f$ \,is intuitionistic fuzzy continuous on \,$U$.\\
\\ $(ii)\;\Rightarrow\;(iii)\;:\;$ Let, $x\;\in\;U$ \,and\, $N$ \,
be a neighbourhood of \,$f(x)$\, in \,$(\,V \,,\, B\,)$.
 Therefore, there exist \,$r\;\in\;(\,0 \,,\, 1 \,)$ \,
 and \,$t\;>\;0$ such that
$B\,(\,(f(x)\,,\,r\,,\,t\,)\,)\;\subset\;N$
$\;\Rightarrow\;x\;\in\;f^{-1}\,(\,B(\,
(f(x)\,,\,r\,,\,t\,)\,)\;\subset\;f^{-1}\,(N).$
\\Again, $\,x\;\in\;f^{-1}\,(\,B\,(\,(\,f(x)\,,\,r\,,\,t\,)\,)\;$ and $\;f^{-1}\,(\,B\,(\,(\,f(x)\,,\,r\,,\,t\,)\,)\;$
is open in \,$(\,U \,,\, A\,)$. So, there exist \,
$r_1\;\in\;(\,0 \,,\, 1\,)$ \,and\, $t_1\;>\;0$ \,such that
\\$B\,(\,x \,,\, r_1 \,,\, t_1 \,)\;\subset\;f^{-1}\,
(\,B\,(\,(\,f(x)\,,\,r\,,\,t\,)\,)
\;\subset\;f^{-1}\,(N)$
\\This shows that $\;f^{-1}\,(N)$ \,is a neighbourhood of
\,$x$ \,in\, $(\,U \,,\, A\,).$\\
\\ $(iii)\;\Rightarrow\;(ii)\;:\;$ Let, $P$ be open in \,$(\,V \,,\, B\,)$
\,and \,$x\;\in\;f^{-1}\,(P)$. \;Then $\,f(x)\;\in\;P$ \,
and therefore there exist \,$\epsilon\,(\;>\;0\,)$
\,and\, $\alpha\;\in\;(\,0 \,,\, 1\,)$ \,such that\\
$B\,(\,f(x)\,,\,\alpha\,,\,\epsilon\,)\;\subset\;P$\\
$\Rightarrow\;\;P\,$ is a neighbourhood of $\,f(x)\,$ in
$\,(\,V \,,\, B\,)$\\
$\Rightarrow\;\;f^{-1}\,(P)\,$ is a neighbourhood of $\,x\,$ in $\,(\,U \,,\, A\,)$
\\$\Rightarrow\;\;\exists\;\delta\;(\;>\;0\;)$ \,and\,
$\beta\;\in\;(\,0\,,\,1\,)$ \,such that
$\,B\,(\,x\,,\,\beta\,,\,\delta\,)\;\subset\;f^{-1}\,(P).$\\
$\Rightarrow\;f^{-1}\,(P)\;$ is open in \,$(\,U\,,\,A\,)$.
\end{proof}
\smallskip
\begin{definition}
$\;f\;:\;U\;\rightarrow\;V$ is said to be \textbf{uniformly intuitionistic
fuzzy continuous} on \,$U$\, if for any given
$\epsilon\;>\;0\;,\;\alpha\;\in\;(\,0 \,,\, 1\,)
\;\exists\;\delta\;=\;\delta\,(\,\alpha \,,\,
\epsilon\,)\;>\;0\;,\;\beta\;=\;\beta\,(\,\alpha \,,\, \epsilon\,)\;>\;0$
\,such that for any two points \, $x_1\,,\,x_2\;\in\;U$,
\[ \hspace{1.0cm} \mu_{\,U}(\,x_1\,-\,x_2 \,,\, \delta \,)
\;>\;1\,-\,\beta\;\;\; and \;\;\;
\nu_{\,U}(\,x_1\,-\,x_2 \,,\, \delta\,)\;<\;\beta\;\;\] \[\Rightarrow\;\mu_{\,V}(\,f(x_1)\,-\,f(x_2),
\epsilon)\;>\;1\,-\,\alpha \;\;\;and \;\;\;
\nu_{\,V}(\,f(x_1)\,-\,f(x_2) \,,\,\epsilon\,)\;<\;\alpha\]
\end{definition}
\smallskip
\begin{theorem}
Let, $f$ be uniformly  intuitionistic fuzzy continuous
on \,$U$. If \,$\{x_n\}_n$ \,is a cauchy sequence in
\,$(\,U\;,\;A\,)$, \,then \,$\{f(x_n)\}_n$ \,is a cauchy sequence
in \,$(\,V\;,\;B\,)$.
\end{theorem}
\begin{proof}
$\;f\;$ is  uniformly intuitionistic fuzzy continuous on \,$U$.
$\\\;\Rightarrow\;$ \,For any given
$\epsilon\;>\;0\;,\;\alpha\;\in\;(\,0\,,\,1\,)\;\exists\;\delta\;=\;
\delta\,(\,\alpha\,,\,\epsilon\,)\;>\;0\;,\;\beta
\,=\,\beta\,(\,\alpha\,,\,\epsilon\,)\;>\;0\;$
such that for any two points $x^\prime\;,\;x^{\prime\prime}\;\in\;U$,
\[\mu_{\,U}(\,x^\prime\,-\,x^{\prime\prime} \,,\,
\delta\,)\;>\;1\,-\,\beta\;\; and \;\;\nu_{\,U}
(\,x^\prime\,-\,x^{\prime\prime} \,,\,
\delta\,)\;<\;\beta\;\;\]
\[\Rightarrow\;\mu_{\,V}(\,f(x^\prime)\,-\,f(x^{\prime\prime}\,) \,,\,
\epsilon\,)\;>\;1\,-\,\alpha \;\;and \;\; \nu_{\,V}(\,f(x^\prime)\,
-\,f(x^{\prime\prime} \,,\,\epsilon\,))\;<\;
\alpha\;\;\;\;\cdots\;\;\;\;(4)\] \\
Since \,$\{x_n\}_n$ \,is a cauchy sequence, for \,
$\delta\;>\;0$ \,and\, $\beta\;\in\;(\,0 \,,\, 1)$
there exist a natural number \,$k$\, such that\\
\[\hspace{2.0cm}\mu_{\,U}(\,x_n\,-\,x_m \,,\, \delta\,)\;>\;1\,-\,\beta
\;\;\;and\;\;\;                                                                                                                             \nu_{\,U}(\,x_n\,-\,x_m \,,\,\delta\,)\;<\;\beta\;\;\;\;\forall\;m,n\;\geq\;k\]
\[\Rightarrow\;\;\mu_{\,U}(\,f(x_n)\,-\,f(x_m) \,,\,\epsilon\,)
\;>\;1\,-\,\alpha \;\;and\;\;                                                                                                                             \nu_{\,U}(\,f(x_n)\,-\,f(x_m) \,,\,\epsilon\,)
\;<\;\alpha\;\;\;\forall\;m,n\;\geq
\;k\;\;\;\;\;(by\;(4))\]\\$\Rightarrow\;\{f(x_n)\}_n$ \,
is a cauchy sequence in\, $(\,V\;,\;B\,)$
\end{proof}
\smallskip
\begin{theorem}
If \,$f\;:\;U\;\rightarrow\;V\;$ is uniformly intuitionistic
fuzzy continuous on $U$ then \,$f$\, is intuitionistic fuzzy
continuous on \,$U$\, but not the converse.
\end{theorem}
\begin{proof}
Obvious.
\end{proof}
\smallskip
To show the converse result does not hold, consider the following example.
\smallskip
\begin{example}
Let, $(\,X\,=\,\mathbb{R}\,,\,\parallel \cdot \parallel\,)$
be a normed linear space, where $\parallel\,x\,\parallel\;=\;\mid\,x\,\mid\;,\;
\forall\;x\;\in\;\mathbb{R}$. Define \,$a\;\ast\;b\;=\;min\;\{\,a\,,\,b\,\}$
\,and\, $a\;\diamond\;b\;=\;max\,\{\,a\,,\,b\,\}\;\;\forall\;a,b\;\in\;
[\,0 \,,\, 1\,]$. Also, define \[\mu_1\;,\;\nu_1\;,\;
\mu_2\;,\;\nu_2\;:\;X\;\times\;
\mathbb{R}\;\rightarrow\;[\,0 \,,\, 1\,] \;\;\;\;by\]
 \[\mu_1\;=\;\frac{t}{t\;+\;|\,x\,|}\;, \;\nu_1\;=\;
 \frac{|\,x\,|}{t\;+\;|\,x\,|}\;,\;
\mu_2\;=\;\frac{t}{t\;+\;k|\,x\,|}\;,
\;\nu_2\;=\;\frac{k|\,x\,|}{t\;+\;k|\,x\,|}\;,\]\\
Let, \,$A\;=\;\{\,(\,(\,x\,,\,t\,)\;,\;\mu_1\;,\;
\nu_1\,)\;:\;(\,x\,,\,t\,)\;\in\;X\;\times\;\mathbb{R}\;\}$ \,and\,
$B\;=\;\{\,(\,(\,x\,,\,t\,)\;,\;\mu_2\;,\;\nu_2\,)
\;:\;(\,x\,,\,t\,)\;\in\;X\;\times\;\mathbb{R}\;\}$
\,be two intuitionistic fuzzy norm on \,$X$.\\
Let us define \,$f(x)\;=\;\frac{1}{x}\;\,\forall\;x\;\in\;
(\,0 \,,\, 1\,)$. \,First we show that \,$f$\,
 is intuitionistic fuzzy continuous on \,$(\,0 \,,\, 1\,)$.
 Let, \,$x_0\;\in\;(\,0\,,\,1\,)$ \,and\, $\{x_n\}_n$ \,
 be a sequence in \,$(\,0\,,\,1\,)$ \,such that\,
 $x_n\;\rightarrow\;x_0$ \,in \,$(\,X\,,\,A\,)$.
i.e.,\;for all\, $t>0$, \[\mathop {\lim }\limits_{n\;
\to \;\infty } \;\mu_1 \,(\,x_n\,-\,x_0 \,,\, t\,)\;\;
=\;\;1\;\;\; and \;\;\;\mathop {\lim }\limits_{n\;
\to \;\infty } \;\nu_1 \,(\,x_n\,-\,x_0 \,,\, t\,)\;\;=\;\;0\;\]
\[\Rightarrow\;\mathop {\lim }\limits_{n\;
\to \;\infty } \;\frac{t}{t\;+\;|\,x_n\,-\,x_0|}\;\;=\;\;
1\;\;\; and \;\;\;\mathop {\lim }\limits_{n\;
\to \;\infty } \;\frac{|\,x_n\,-\,x_0\,|}{t\;+\;|\,x_n\,
-\,x_0|\,}\;\;=\;\;0\;\]\\
\[\Rightarrow\;\mathop {\lim }\limits_{n\;
\to \;\infty } \;|\,x_n\,-\,x_0\,|\;\;=\;\;0\hspace{7.5cm}\]\\
Again, for all \,$t>0$,\,
\[\mu_2\,(\,f(x_n)\,-\,f(x_0)\,,\,t\,)\;\;=\;\;
\frac{t}{t\;+\;k\,|\,f(x_n)\,-\,f(x_0)\,|}\;\,,
\;\;=\;\;\frac{t \;x_n\; x_0}{t\; x_n\; x_0\;+\;k\,|\,x_n\,-\,x_0\,|}\]
\[\Rightarrow\;\mathop {\lim }\limits_{n\;
\to\;\infty } \;\mu_2(\,f(x_n)\,-\,f(x_0)\,,\,t\,)
\;\,=\,\;1 \hspace{1.5cm}\] and
\[ \;\;\nu_{\,2}(\,f(x_n)\,-\,
f(x_0) \,,\, t\,)\;=\;\frac{k\,|\,f(x_n)\,-\,f(x_0)\,|}{t\;+
\;k\,|\,f(x_n)\,-\,f(x_0)\,|}\;\;=\;\;
\frac{k\,|\,x_n\,-\,x_0\,|}{t\, x_n\, x_0\;+\;k\,|\,x_n\,-\,x_0|\,}\]
\[\Rightarrow\;\mathop {\lim }\limits_{n\;
\to \;\infty } \;\nu_{\,2}(\,f(x_n)\,-\,f(x_0) \,,\,t\,)
\;\;=\;\;0 \hspace{1.5cm}\] \\
Thus \,$f$\, is sequentially intuitionistic fuzzy continuous
on \,$(\,0\,,\,1\,)$\, and hence intuitionistic fuzzy continuous
on \,$(\,0\,,\,1\,)$. We now show that \,$f$\, is not
uniformly intuitionistic fuzzy continuous on \,$(\,0\,,\,1\,)$.
By example $2$ of \cite{Samanta}, we see that \,$\{x_n\}_n$\,
is a cauchy sequence in \,$(\,X\;,\;\|\,\cdot\,\|\,)$\, if and only if
\,$\{x_n\}_n$\, is a cauchy sequence in \,$(\,X\;,\;A\,)$
\;or\; $(\;X\;,\;B\;)$. \\ Let, \,$x_n\;=\;\frac{1}{n\,+\,1}\;\;\forall\;n\;\in\;\mathbb{N}$.
So, \,$\{f(x_n)\}_n$\, is a not a cauchy sequence in \,$(\,X\;,\;\|\,\cdot\,\|\,)$
and hence not a cauchy sequence \,$(\,X\,,\,B\,)$.\\
Consequently, \,$f$\, is not uniformly intuitionistic fuzzy
continuous on \,$(\,0\,,\,1\,)$.
\end{example}
\bigskip
\section{\textbf{Uniformly Intuitionistic Fuzzy Convergence}}
 In this section we assume that \,$(\,U\;,\;A\,)$ \,and\,
(\,V\;,\;B\,) \,are two intuitionistic fuzzy normed linear space over
the same field \,$F$.
\begin{definition}
Let, \,$f_n\;:\;(\,U\;,\;A\,)\;\rightarrow\;(\,V\;,\;B\,)\;$\, be a
sequence of functions. The sequence $\;\{f_n\}_n\;$ is said to be
\textbf{pointwise intuitionistic fuzzy convergent} on \,$U$\,
with respect to \,$A$\, if for each $\;x\;\in\;U\;$
, the sequence \,$\;\{\,f_n\,(\,x\,)\,\}_n\;$\,
is convergent with respect to \,$B$.
\end{definition}
Let, the sequence \,$\{f_n\}_n$\, be pointwise
intuitionistic fuzzy convergent on \,$U$\, and let, \,$c\;\in \;U$.\,
Then the sequence \,$\{\,f_n\,(\,c\,)\,\}_n$\, is intuitionistic fuzzy
convergent on \,$(\,V\;,\;B\,)$. Let, \,$f_n\,(\,c\,)
\;\rightarrow\;y_c$\, in \,$(\,V\;,\;B\,)$. \,Then \,$y_c$\,
is unique. Let us now define \,$f\;:\;(\,U\,,\,A\,)\;\rightarrow\;
(\,V\;,\;B\;\,)$\, by \,$f\,(x)\;=\;y_{\,x}\;\;\forall\;x\;\in\;U$, where
$f_n\,(\,x\,)\;\rightarrow\;y_{\,x}$\, in \,$(\,V\;,\;B\,)$.
\,Then \,$f$\, is said to be the intuitionistic fuzzy limit
function of the sequence \,$\{f_n\}_n$\, on \,$U$\, and it is written as \,$f_n\;\rightarrow\;f$\, on \,$(\;U\;,\;A\;)$.
\smallskip
\begin{example}
Let, \,$a\;\ast\;b\;=\;min\;\{\,a\;,\;b\,\}\;,\;a\;\diamond\;b\;
=\;max\;\{\,a\;,\;b\,\}$\, for all \,$a\,,\,b\;\in\;[\,0\,,\,1\,].$\, Define \,$\mu\,(\,x\,,\,t\,)\;=\;\frac{t}{t\;+\;|\,x\,|}$\,
\,and\, \,$\nu\,(\,x\,,\,t\,)\;=\;\frac{|\,x\,|}{t\;+\;|\,x\,|}.$
\\Let, \,$U\;=\;(\,-1 \,,\, 1\,)$\, , \,
$V\;=\;\mathbb{R}$\,,\,$\mu \;=\; \mu_{\,U} \;=\; \mu_{\,V}$ \,,\,
\,$\nu \;=\; \nu_{\,U} \;=\; \nu_{\,V}$\, and \,$\;f
_n\;:\;(\,U \,,\, A\,)\;\rightarrow\;(\,V \,,\, B\,)$\, be defined by \,$f_n\,(x)\;=\;x^{\,n}\;\;\forall\;x\;\in\;U$.\, Also, let \,$O\,(x)\;=\;0\;\;\forall\;x\;\in\;U$. \,Therefore,
\[\mu\,(\,f_n\,(x)
\,-\, 0 \;,\; t\,)\;\;=\;\;\frac{t}
{t\;+\;|\,x\,|^{\,n}}\;\longrightarrow\;1\;\;\;
\;as \;\;\;n\;\rightarrow\;\infty\;\]
and \[\nu\,(\,f_n\,(x)
\,-\,0\,,\,t\,)\;=\;\frac{|\,x\,|^{\,n}}{t\;+\;|\,x\,|^{\,n}}
\;=\;1\,-\,\frac{t}{t\;
+\;|\,x\,|^{\,n}}\;\rightarrow\;0\;\;\; as\;\; \;n\;\rightarrow\;\infty\]\\
$\Rightarrow\;\;\{f_n\}_n$\, is pointwise intuitionistic fuzzy
convergent to \,$O$\, on \,$(\,U \,,\, A\,).$
\end{example}
\smallskip
\begin{example}
Let, \,$a\;\ast\;b\;=\;min\;\{\,a\;,\;b\,\}\;,\;a\;\diamond\;b\;
=\;max\;\{\,a\;,\;b\,\}$\, for all \,$a\,,\,b\;\in\;[\,0\,,\,1\,].$\,
Let, \,$U\;=\;\{\,x\;\in\;\mathbb{R}\;:\;x\;\geq\;0 \,\}$ \,,\,
$V=\mathbb{R}$\,,\,\,$\mu \;=\; \mu_{\,U} \;=\; \mu_{\,V}$ \,,\,
\,$\nu \;=\; \nu_{\,U} \;=\; \nu_{\,V}$\, \;where \[\mu\,(\,x\,,\,t\,)\;=
\;\frac{t}{t\;+\;|\,x\,|}\,
\;\;\;and\;\;\; \,\nu\,(\,x\,,\,t\,)\;=\;\frac{|\,x\,|}{t\;+\;|\,x\,|}.\]
Consider, \[g_n\,(x)\;=\;\frac{n}{x\;+\;n}\;\;\forall\;x\;\in\;U
\;\;\;and\;\;\; g\,(x)\;\;=\;\;1\;\;\forall\;x\;\in\;U.\]
\[Therefore, \;\;\;g_n\,(x)\,-\,g\,(x)\;\;=\;\;
\frac{n}{x\;+\;n}\;-\;1\;\;=\;\;-\;\frac{x}{x\;+\;n}\]
\[\mu\,(\,g_n\,(x)\,-\,g\,(x)\;,\;t\,)\;\,=\;\,\mu\,(\,-\;\frac{x}{x
\;+\;n}\;,\;t\,) \hspace{4.5cm}\] \[ \hspace{3.5cm}=\;\;\frac{t}{t\;
+\;|\;-\;\frac{x}{x\;+\;n}\;|}\;\;
=\;\;\frac{t}{t\;+\;\frac{x}{x\;+\;n}\;}\;\rightarrow
\;1\;\; as\; \;n\;\rightarrow\;\infty\]
and \[\nu\,(\,g_n\,(x)\,-\,g\,(x)\,,\,t\,)\;=\;\frac{\frac{x}{x\;
+\;n}}{t\;+\;\frac{x}{x\;+\;n}}\;=\;\frac{x}{x\;+
\;t\,(\,x\,+\,n\,)}\;\rightarrow\;0\;\, as \;\;n\;
\rightarrow\;\infty \]\\Thus, we see that \,$g_n(x)\;\rightarrow \;g(x)\;\;\;\forall\;x\;\in\;U$\, with respect to \,$B.$
\end{example}
\smallskip
\begin{definition}
Let, \,$f_n\;\:\;(\,U\;,\;A\,)\rightarrow\;(\,V\;,\;B\,)$ \,be a
sequence of functions.The sequence \,$\{f_n\}_n$\, is said to be
\textbf{uniformly intuitionistic fuzzy convergent} on \,$U$\, to a
function $f$ with respect to \,$A$, \,if given
\,$0\;<\;r\;<\;1\;,\;t\;>\;0$\, there exist a positive integer\,
$n_0\;=\;n_0\;(\,r\,,\,t\,)$\, such that \,$\forall\;x\;\in\;U$\, and
\,$\forall\;n\geq n_0\;,$
\[\mu\,(\,f_n(x)\;-\;f(x) \,,\, t\,)\;>\;1\,-\,r
\;\;,\; \;\nu\,(\,f_n(x)\;-\;f(x) \,,\, t\,)\;<\;r\]
\end{definition}
\smallskip
\begin{example}
\;In the example\,$(4.1)$, we have seen that \,$\;f_n\;\rightarrow\;O\;$\,
with respect to \,$A$. Let us show that this convergence is not
uniform on \,$(\,0\;,\;1\,)$\, but converges uniformly on
\,$[\,0 \,,\, a\,]$\, \,\, where \,$0\;<\;a\;<\;1$\,,\,
with respect to \,$A$. \\Let, \,$c\;\in\;(\,0\,,\,1\,)\;,
\;r\;\in\;(\,0\,,\,1\,)$\,
and \,$\;t\;>\;0\;.$ \,Then,
\[\mu\,(\,f_n(c)\,-\,O\,(c)\;,\;t\,)\;>\;1\,-\,r\;\;\;and \;\;\;\nu\,(\,f_n(c)\,-\,O\,(c)\;,\;t\,)\;<\;r\]
\[\Rightarrow\;\;\frac{t}{t\;+\;c^{\,n}}\;>\;1\,-\,r \;\;\;and\;\;\; \;\frac{c^{\,n}}{t\;+\;c^{\,n}}\;<\;r\]
\[\Rightarrow\;\;c^{\,n}\;<\;\frac{r\,t}{(\,1\,-\,r\,)} \;\;\;\Rightarrow\;\;\frac{1}{c^{\,n}}\;>\;\frac{(\,1\,-\,r\,)}{r\,t}\;\]
\[\;\;\;\;\Rightarrow\;\;n\;>\;\frac{\log\;\left(\,\frac{(\,1\,-\,r\,)}{r\,t}\right)}
{\log\,\left(\,\frac{1}{c}\,\right)} \hspace{4.5cm}\]
\\Let, \,$k\;\; = \;\;\left[ {\;\frac{{\log \;\left(
{\,\frac{{1\; - \;r}}{{r\;t}}\,} \right)}}{{\log \,
\left( {\,\frac{1}{c}\,} \right)}}\;} \right]\;\; + \;\;1$\\\\
Then, for each \,$x\;\in(\,0\,,\,1\,)$\, and given \,
$r\;\in(\,0\,,\,1\,)$\, and \,$t\;>\;0\,,$
\[\mu\,(\,f_n(x)\,-\,O(x)\;,\;t\,)\;>\;1\,-\,r\;\;and \;\;\nu\,(\,f_n(x)\,-\,O(x)\;,\;t\,)\;<\;r\;\;\;\forall\;n\;\geq\;k\]
where, \,$k\;\; = \;\;\left[ {\;\frac{{\log \;\left(
{\,\frac{{1\; - \;r}}{{r\;t}}\,} \right)}}{{\log \,
\left( {\,\frac{1}{x}\,} \right)}}\;} \right]\; + \;1$\,,\, which shows
that \,$k$\, depends on \,$r\,,\,t$\, as well as on \,$x$.\,
 Also, we see that as \,$x\;\rightarrow\;1\;\Rightarrow
 \;k\;\rightarrow\;\infty.$\\
$\Rightarrow\;\;\{\,f_n\,\}_n$\, is not uniformly
 intuitionistic fuzzy convergent on \,$(\,0\,,\,1\,)$\,
 with respect to \,$A.$\\
Let, \,$a\;\in\;(\,0\,,\,1\,)$.\, In \,$[\,0\,,\,a\,]$, \,the greatest value of \,
$\left[ {\;\frac{{\log \;\left(
{\,\frac{{1\; - \;r}}{{r\;t}}\,} \right)}}{{\log \,
\left( {\,\frac{1}{x}\,} \right)}}\;} \right]$\; is\;
 $\left[ {\;\frac{{\log \;\left(
{\,\frac{{1\; - \;r}}{{r\;t}}\,} \right)}}{{\log \,
\left( {\,\frac{1}{a}\,} \right)}}\;} \right]$. \,So,
let $\;n_0\;=\;\left[ {\;\frac{{\log \;\left(
{\,\frac{{1\; - \;r}}{{r\;t}}\,} \right)}}{{\log \,
\left( {\,\frac{1}{a}\,} \right)}}\;} \right]\;+\;1.$
\\\\Therefore, for all \,$x\;\in\;[\,0\,,\,a\,]$\,,\,
given \,$r\;\in\;(\,0\,,\,1\,)$\,
and \,$t\;>\;0$\,,\, there exist a natural number
\,$n_0\;=\;n_0(\,r\,,\,t\,)$\, such that
\[\mu\,(\,f_n(x)\,-\,O(x)\;,\;t\,)\;>\;1\,-\,r\;\; and \;\;\nu\,(\,f_n(x)\,-\,O(x)\;,\;t\,)\;<\;r\;\;\;\forall\;n\;\geq\;n_0\]
$\Rightarrow\;\;\;\{\,f_n\,\}_n$\, is uniformly
intuitionistic fuzzy  convergent on \,$[\,0\,,\,a\,]$\, with respect to
\,$A$\, , \,where \,$a\;\in\;(\,0\,,\,1\,)$.
\end{example}
\smallskip
\begin{result}
Let \,$ \left( {\,U\;,\;\left\| {\, \cdot \,} \right\|_{\,1} \,} \right) $
\,and\, $ \left( {\,V\;,\;\left\| {\, \cdot \,} \right\|_{\,2} \,} \right) $\,
be two normed linear space over the field \,$ K \;=\; \mathbb{R}
\;or\; \mathbb{C}$\, , \,$ f_{\,n} \;:\; U \;\rightarrow\; V \;\,
\forall \;\,n \;\in\; \mathbb{N}$\, , \,$ a \;\ast\; b \;=\; \min\,\{\,a
\,,\, b\,\} $, \,$ a \;\diamond\; b \;=\; \max\,\{\,a
\,,\, b\,\} \;\; \forall \;\, a \,,\, b \;\in\; [\,0 \,,\, 1\,]$. For all
\,$t \;>\; 0$\, , \,define \[\mu_{\,U}\,(\,x \,,\, t\,) \;=\; \frac{t}
{t \,+\, k\,\left\| {\, x \,} \right\|_{\,1}} \;\;\;,\;\;\;
 \nu_{\,U}\,(\,x \,,\, t\,) \;=\; \frac{k\,\left\| {\, x \,} \right\|_{\,1}}
{t \,+\, k\,\left\| {\, x \,} \right\|_{\,1}} \;\,,\]
\[\mu_{\,V}\,(\,x \,,\, t\,) \;=\; \frac{t}
{t \,+\, k\,\left\| {\, x \,} \right\|_{\,2}} \;\;\;,\;\;\;
 \nu_{\,V}\,(\,x \,,\, t\,) \;=\; \frac{k\,\left\| {\, x \,} \right\|_{\,2}}
{t \,+\, k\,\left\| {\, x \,} \right\|_{\,2}} \;\,,\] where \,$k \;>\; 0$ .
Let \[A \;\,=\;\, \left\{\,\left(\,(\,x \,,\, t\,) \,,\, \mu_{\,U}\,(\,x
\,,\, t\,) \,,\, \nu_{\,U}\,(\,x \,,\, t\,)\,\right) \;\,:\;\, (\,x \,,\, t\,) \;
\in\; U\;\times\;\mathbb{R^{\,+}}\,\right\} \;,\]
\[B \;\,=\;\, \left\{\,\left(\,(\,x \,,\, t\,) \,,\, \mu_{\,V}\,(\,x
\,,\, t\,) \,,\, \nu_{\,V}\,(\,x \,,\, t\,)\,\right) \;\,:\;\, (\,x \,,\, t\,) \;
\in\; U\;\times\;\mathbb{R^{\,+}}\,\right\} \;\,\,\,\]
Then \,$(\,U \,,\, A\,)$\, and \,$(\,V \,,\, B\,)$\, are intuitionistic fuzzy
normed linear space. Following the example $(2)$ of \cite{Samanta} , it can
shown that \,$\{\,f_{\,n}\,\}$\, is uniformly intuitionistic fuzzy convergent on
\,$U$\, with respect to \,$A$\, if and only if \,$\{\,f_{\,n}\,\}$\, is uniformly
convergent with respect to \,$\left\| {\, \cdot \,} \right\|_{\,1}$ .
\end{result}
\begin{theorem}
Let, \,$f_n\;:\;(\,U\,,\,A\,)\;\rightarrow\;(\,V\,,\,B\,)
\;,\;\forall\;n\;\in\;\mathbb{N}$ \,be a sequence of functions.
Then the sequence \,$\{\,f_n\,\}_n$\, is uniformly intuitionistic
fuzzy  convergent on \,$(\,U\,,\,A\,)$\, if and only if for any
given \,$r\;\in\;(\,0\,,\,1\,)$\, and \,$t\;>\;0$\, there exist
a natural number \,$k\;=\;k\,(\,r\;,\;t\,)$\, such that
\,$\forall\;x\;\in \;U$\,,\[\mu\,(\,f_{n + p}(x)\,-\,f_n(x)
\;,\;t\,)\;>\;1\,-\,r \;
\;,\;\;\nu\,(\,f_{n + p}(x)\,-\,f_n(x)\;,\;t\,)\;<\;r\;\;,\]
\[\hspace{5.5cm}\forall\;n\;
\geq\;k\;and\;p\;=\;1\,,\,2\,,\,3\,,\,\cdots\]
\end{theorem}
\begin{proof}
\textbf{$\Rightarrow\;$part\;:\;}\,Let, \,$\{\,f_n\,\}_n$\,
be uniformly intuitionistic fuzzy convergent on \,$(\,U\,,\,A\,)$\,
 and \,$f$\, be its limit function. Then for any given
\,$r\;\in\;(\,0\,,\,1\,)$\, and \,$t\;>\;0$\, there exist a natural number \,$n_0\;=\;n_0\,(\,r\;,\;t\,)$\, such that for all \,$x\;\in\;U$\,, and \,$\forall\;n\;\geq\;n_0$ ,
\[\hspace{1.5cm}\mu\,\left(\,f_n(x)\;-\;f(x)\;,\;\frac{t}{2}
\,\right)\;>\;1\,-\,r\;\;,\;\;\nu
\,\left(\,f_n(x)\;-\;f(x)\;,\;\frac{t}{2}\,\right)\;<\;r\]
\\$\Rightarrow\;$ For all $\;n\;\geq\;n_0\;$ and
$\;p\;=\;1\,,\,2\,,\,3\,,\;\cdots\;$ and $\;x\;\in\;U\;$, \[\hspace{1.5cm}\mu\,\left(\,f_{n\,+\,p}(x)\;-\;f(x)\;,\;\frac{t}{2}\,\right)
\;>\;1\,-\,r\;\;,\;\;\nu\,\left(\,f_{n\,+\,p}(x)\;-\;f(x)\;,\;
\frac{t}{2}\,\right)\;<\;r\]
\\Now, for all $\;x\;\in\;D\;$ and $\;p\;=\;1\,,\,2\,,\,3\,,
\;\cdots\;,\;$,\, we see that\\
\[\mu\,\left(\,f_{n\,+\,p}\,(x)\;-\;f_n\,(x)\;,\;t\,\right)\hspace{6.2cm}\]
\[\hspace{1.5cm}=\;\mu\,\left(\,f_{n\,+\,p}
\,(x)\;-\;f\,(x)\;+\;f\,(x)\;-\;f_n\,(x)\;\;,\;\;\frac{t}{2}\;
+\;\frac{t}{2}\,\right)\]
\[\hspace{1.0cm}\geq\;\mu\,\left(\,f_{n\,+\,p}\,(x)\;-\;f\,(x)\;
,\;\frac{t}{2}\,\right)\;\ast\;
\mu\,\left(\,f\,(x)\;-\;f_n\,(x)\;,\;\frac{t}{2}\,\right)\]
\[\hspace{1.0cm}=\;\mu\,\left(\,f_{n\,+\,p}\,(x)\;-\;f\;(x)
\;,\;\frac{t}{2}\,\right)\;\ast\;\mu
\,\left(\,f_n\,(x)\,-\,f\,(x)\;,\;\frac{t}{2}\,\right)\]
\[>\;(\,1\;-\;r\,)\;\ast\;(\,1\;-\;r\,)\;\;
=\;\;(\,1\;-\;r\,)
\;\;\;\;\;\;\;\forall\;n\;\geq\;n_0\] and \[\nu\,\left(\,f_{n\,+\,p}\,(x)\;-\;f_n\,(x)\;,\;t\,\right)\hspace{6.2cm}\]
\[\hspace{1.5cm}=\;\nu\,\left(\,f_{n\,+\,p}
\,(x)\;-\;f\,(x)\;+\;f\,(x)\;-\;f_n\,(x)\;\;,\;\;\frac{t}{2}\;
+\;\frac{t}{2}\,\right)\]
\[\hspace{1.7cm}\leq\;\nu\,\left(\,f_{n\,+\,p}\,(x)\;-\;f\,(x)\;
,\;\frac{t}{2}\,\right)\;\diamond\;
\nu\,\left(\,f\,(x)\;-\;f_n\,(x)\;,\;\frac{t}{2}\,\right)\]
\[\hspace{1.6cm}=\;\nu\,\left(\,f_{n\,+\,p}\,(x)\;-\;f\;(x)
\;,\;\frac{t}{2}\,\right)\;\diamond\;\nu
\,\left(\,f_n\,(x)\,-\,f\,(x)\;,\;\frac{t}{2}\,\right)\]
\[<\;r\;\diamond\;r\;\;
=\;\;r \;\;\;\;\;\;\;\forall\;n\;\geq\;n_0 \hspace{3.5cm}\]
\\Hence the $\;\Rightarrow\;$ part.\\\\
\textbf{$\Leftarrow\;$part\;:\;} In this part, we suppose that for any given \,$r\;\in\;(\,0\,,\,1\,)$\, and \,$t\;>\;0$\, there exist a natural number \,$n_0\;=\;n_0\,(\,r\;,\;t\,)$\, such that for all \,$x\;\in\;U$\, and \,$\forall\;n\;\geq\;n_0$
\[\mu\,(\,f_{n\,+\,p}\,(x)\;-\;f_n\,(x)\;,\;t\,)\;>\;1\,-\,r
\;\;,\;\;\nu\,
(\,f_{n\,+\,p}\,(x)\;-\;f_n\,(x)\;,\;t\,)\;<\;r.\] Let
\,$x_0\;\in\;U.$\, Then for \,$\forall\;n\;\geq\;n_0$\, we see that,
\[\mu\,(\,f_{n\,+\,p}\,(x_0)\;-\;f_n\,(x_0)\;,\;t\,)\;>\;1\,-\,r\;\;,\;\;
\nu\,(\,f_{n\,+\,p}\,(x_0)\;-\;f_n\,(x_0)\;,\;t\,)\;<\;r.\]
$\Rightarrow\;\;\{\,f_n(x_0)\,\}_n$\, is an intuitionistic fuzzy cauchy sequence in \,$\left(\,V\,,\,B\,\right).$\\
\;$\Rightarrow\;\;\{\,f_n(x_0)\,\}_n$\, is an intuitionistic fuzzy convergent in \,$\left(\,V\,,\,B\,\right).$ \\
\;$\Rightarrow\;\;\{\,f_n\,\}_n$\, is pointwise intuitionistic fuzzy convergent on $\left(\,U\,,\,A\,\right).$ \\\\
Let, \,$f$\, be the intuitionistic fuzzy limit function of \,$\{\,f_n\,\}_n$\,
on $\left(\,U\,,\,A\,\right).$ Let, \,$r\;\in\;(\,0\,,\,1\,)$\,
and \,$t\;>\;0$.\, Then by the given condition, there exist a natural number \,$n_0\;=\;n_0\,(\,r\,,\,t\,)$\, such that for all \,$x\;\in\;U$\, and \,$p\;=\;1\,,\,2\,,\,3\,,\;\cdots$\, and \,$\forall\;n\;\geq\;n_0$\,
\[\mu\,\left(\,f_{n\,+\,p}\,(x)\;-\;f_n\,(x)\;,\;\frac{t}{2}\,\right)
\;>\;1\,-\,r\;\;,
\;\;\nu\,\left(\,f_{n\,+\,p}\,(x)\;-\;f_n\,(x)\;,\;\frac{t}{2}\,\right)\;<\;r.\]
\\Again since \,$f_n\;\rightarrow\;f$\, as \,$n\;\rightarrow\;\infty$\,
 on \,$\left(\,U\,,\,A\,\right)$\,,\, we see that
 \,$f_{n\,+\,p}\;\rightarrow\;f$\, as \,$n\;\rightarrow\;\infty$ on $\left(\,U\,,\,A\,\right)$\,,\,which implies that
for all \,$n\;\geq\;n_0$\, and for all \,$x\;\in\;U$\,, \[\mu\,\left(\,f_{n\,+\,p}\,(x)\;-\;f\,(x)\;,\;\frac{t}{2}\,\right)
\;>\;1\,-\,r\;\;,\;
\;\nu\;\left(\,f_{n\,+\,p}\,(x)\;-\;f\,(x)\;,\;\frac{t}{2}\,\right)\;<\;r\]
\\Now\,,\, for all $\;x\;\in\;U\;$ we see that\\
\[\mu\,\left(\,f_n\,(x)\;-\;f\,(x)\;,\;t\,\right)\hspace{7.5cm}\]
\[\hspace{2.5cm}=\;\;\mu\,\left(\,f_n\,(x)\;-\;f_{n\,+\,p}\,(x)
\;+\;f_{n\,+\,p}\,(x)\;-\;f\,(x)\;,\;\frac{t}{2}\;+\;\frac{t}{2}\,\right)\]
\[\hspace{2.9cm}\geq\;\mu\,\left(\,f_n\,(x)\;-\;f_{n\,+\,p}\,(x)\;,\;
\frac{t}{2}\,\right)\;\ast\;\mu
\,\left(\,f_{n\,+\,p}\,(x)\;-\;f\,(x)\;,\;\frac{t}{2}\,\right)\]
\[\hspace{1.2cm}>\;(\,1\,-\,r\,)\;\ast\;(\,1\,-\,r\,)\;=\;(\,1\,-\,r\,)\;
\;\;,\;\;\;\;\;\forall\;n\;\geq\;n_0\]
and \[\nu\,\left(\,f_n\,(x)\;-\;f\,(x)\;,\;t\,\right)\hspace{7.5cm}\]
\[\hspace{2.5cm}=\;\;\nu\,\left(\,f_n\,(x)\;-\;f_{n\,+\,p}\,(x)
\;+\;f_{n\,+\,p}\,(x)\;-\;f\,(x)\;,\;\frac{t}{2}\;+\;\frac{t}{2}\,\right)\]
\[\hspace{2.9cm}\leq\;\nu\,\left(\,f_n\,(x)\;-\;f_{n\,+\,p}\,(x)\;,\;
\frac{t}{2}\,\right)\;\diamond\;\nu
\,\left(\,f_{n\,+\,p}\,(x)\;-\;f\,(x)\;,\;\frac{t}{2}\,\right)\]
\[<\;r\;\diamond\;r\;=\;r\;
\;\;,\;\;\;\;\;\forall\;n\;\geq\;n_0 \hspace{2.8cm}\] \\
$\Rightarrow\;\;\{\,f_n\}_n$\, is uniformly
intuitionistic fuzzy convergent on \,$\left(\,U\,,\,A\,\right)$.
\end{proof}
\smallskip
\textbf{Equivalent Statement:\;\;}Let, \,$f_n\;:\;(\,U\,,\,A\,)\;\rightarrow\;(\,V\,,\,B\,)
\;,\;\forall\;n\;\in\;\mathbb{N}$\, be a sequence of functions. Then the sequence \,$\{\,f_n\,\}_n$\, is uniformly intuitionistic fuzzy convergent on \,$\left(\,U\,,\,A\,\right)$\, if and only if for any given
\,$r\;\in\;(\,0\,,\,1\,)$\, and \,$t\;>\;0$\, there exist  a natural number \,$n_0\;=\;n_0\,(\,r\;,\;t\,)$\, such that \,$\forall\;x\;\in\;U$,\[\mu\,(\,f_{n}\,(x)\;-\;f_m(x)\;,\;t\,)\;>\;1\,-\,r\;\;,
\;\;\nu\,(\,f_{n}\,(x)\,-\,f_m\,(x)\;,\;t\,)\;<\;r\;\;,
\;\;\forall\;\;n\,,\,m\;\geq\;n_0.\]

\begin{example}
In the example $(\,4.3\,)$, though we have seen that \,$\{\,f_n\,\}_n$\,
is uniformly intuitionistic fuzzy convergent on \,$[\,0\,,\,a\,]$\,,\,
where \,$a\;\in\;(\,0\,,\,1\,)$\, and \,$f_n\,(x)\;=\;x^n$\,,\, again, we will
verify it by using the above theorem .
Let, \,$r\;\in\;(\,0\,,\,1\,)$\, and \,$t\;>\;0$.\,
Again let, \,$m\,,\,n\;\in\;\mathbb{N}$\, such that \,$m\;<\;n$\,.\,
Now ,
\[\mu\,(\,f_n\,(x)\;-\;f_m\,(x)\;,\;t\,)\;>\;1\,-\,r\;\;,\;\;\;\nu\,
(\,f_n\,(x)\;-\;f_m\,(x)\;,\;t\,)\;<\;r\]
\[\hspace{2.0cm}\Rightarrow\;\;\;\mu\,(\,x^n\;-\;x^m\;,\;t\,)\;>\;
1\,-\,r\;\;,\;\;\;\nu\,(\,x^n\;-\;x^m\;,
\;t\,)\;<\;r\] \[\Rightarrow\;\;\mid\;x^n\;-\;x^m\;\mid\;<
\;\frac{r\,t}{(1\,-\,r)} \;.\hspace{3.2cm}\]
\\Since,
$\mathop {\sup }\limits_{x\; \in \;[\,0\;,\;a\,]}
\,\left| {\,x^{\,n} \; - \;x^{\,m} \,} \right|\;\,
= \;\,2\,a^{\,m} \;\,,\;\,m\;\, < \;\,n$
 we have \,,\, $2\;a^m\;<\;\frac{r\;t}{(\,1 \,-\,r\,)}$ \,,\,
 which implies that\,
$m\;>\; \left[ {\;\frac{{\log \;2\;\left( {\,\frac{{1\;
- \;r}}{{r\;t}}\,} \right)}}{{\log \,\left( {\,\frac{1}{a}\,}
\right)}}\;} \right]$
\,Let \,,\, $\;\;k\;=\;\left[ {\;\frac{{\log \;2\;\left( {\,\frac{{1\;
- \;r}}{{r\;t}}\,} \right)}}{{\log \,\left( {\,\frac{1}{a}\,}
\right)}}\;} \right]\;+\;1\;.$ \,Thus, we see that for given
\,$r\;\in\;(\,0\,,\,1\,)$\, and \,$t\;>\;0$\,,\, there exist a natural number \,$k\;=\;k\,(\,r\,,\,t\,)$\, such that \,$\forall\;x\;\in\;[\,0\,,\,a\,]\;,\;a\;\in\;(\,0\,,\,1\,)$\, and \,$\forall\;n\;>\;m\;\geq\;k$ \[\mu\,(\,f_n\,(x)\;-\;f_m\,(x)\;,\;t\,)
\;>\;1\,-\,r\;\;,\;\;\nu\,(\,f_n\,(x)\;-\;f_m\,(x)\;,\;t\;)\;<\;r\;.\]
This completes the verification .
\end{example}
\smallskip
\begin{theorem}
\textbf{(Uniform Limit Theorem)}: Let, \,$(\,U\,,\,A\,)$\,
and \,$(\,V\,,\,B\,)$\, be two intuitionistic fuzzy normed
linear space satisfying the condition \,$(xii)$\,. Also, let \,$f_n\;:\;(\,U\,,\,A\,)\;\rightarrow\;(\,V\,,\,B\,)\;\;,
\;\;\forall\;n\;\in\;\mathbb{N}$\, and \,$f_n$\,
be intuitionistic fuzzy continuous on \,$(\,U\,,\,A\,)$\,.\,
If \,$\{\,f_n\,\}_n$\, be uniformly intuitionistic fuzzy
convergent on \,$(\,U\,,\,A\,)$\, to a function \,$f$\, then \,$f$\,
is intuitionistic fuzzy continuous on \,$(\,U\,,\,A\,)$\,.
\end{theorem}
\begin{proof}
Let \,$\{\,f_{\,n}\,\}_{\,n}$\, be uniformly intuitionistic fuzzy
convergent to the function \,$f$\, on \,$(\,U\,,\,A\,)$ .\,Then for any given
\,$r\;\in\;(\,0\,,\,1\,)$\, and \,$t\;>\;0$, there exists a natural number
\,$k \;=\; k\,(\,r \,,\, t\,)$\, such that for all \,$x \;\in\;U$\, and for all
\,$n\;\geq\;k$\,, \[ \mu_{\,V}\,\left(\,f_{\,n}\,( x ) \;-\; f( x ) \,,\,\frac{t}{3}\right)
\;\;>\;\; 1 \;-\; r \;\;,\;\; \nu_{\,V}\,\left(\,f_{\,n}\,( x ) \;-\;
f( x ) \,,\,\frac{t}{3}\right)
\;\;<\;\;  r \] Thus, for all \,$x \;\in\;U$ ,
\[ \mu_{\,V}\,\left(\,f_{\,k}\,( x ) \;-\; f( x ) \,,\,\frac{t}{3}\right)
\;\;>\;\; 1 \;-\; r \;\;,\;\; \nu_{\,V}\,\left(\,f_{\,k}\,( x ) \;-\;
f( x ) \,,\,\frac{t}{3}\right)
\;\;<\;\;  r \] Let \,$x_{\,0}$\, be an arbitrary but fixed point of
\,$U$\,. Then we have \[ \mu_{\,V}\,\left(\,f_{\,k}\,( x_{\,0} ) \;-\; f( x_{\,0} ) \,,
\,\frac{t}{3}\right)
\;\;>\;\; 1 \;-\; r \;\;,\;\; \nu_{\,V}\,\left(\,f_{\,k}\,( x_{\,0} ) \;-\;
f( x_{\,0} ) \,,\,\frac{t}{3}\right)
\;\;<\;\;  r \] Since each \,$f_{\,n}$\, is intuitionistic fuzzy continuous
on \,$U$ ,\,$f_{\,k}$\, is intuitionistic fuzzy continuous at \,$x_{\,0}$ .
So, for any given \,$r\;\in\;(\,0\,,\,1\,)$\, and \,$t\;>\;0$, \,there exist
\,$\delta \;=\; \delta\,\left(\,r \,,\, \frac{t}{3}\,\right) \;>\; 0$ \,,\,
$\beta \;=\; \beta\,\left(\,r \,,\, \frac{t}{3}\,\right) \;\in\;(\,0\,,\,1\,)$\,
such that \[ \mu_{\,U}\,(\,x \,-\,x_{\,0} \;,\; \delta\,) \;\,>\;\,
1 \,-\, \beta \;\,\Rightarrow
\;\,\mu_{\,V}\,\left(\,f_{\,k}(x) \,-\,f_{\,k}(\,x_{\,0}\,) \;,\; \frac{t}{3}
\,\right) \;\,>\;\,1 \,-\, r \; ,\]  \[ \nu_{\,U}\,(\,x \,-\,x_{\,0} \;,\; \delta\,) \;\,<\;\,\beta \;\,\Rightarrow
\;\,\nu_{\,V}\,\left(\,f_{\,k}(x) \,-\,f_{\,k}(\,x_{\,0}\,) \;,\; \frac{t}{3}
\,\right) \;\,<\;\, r \hspace{0.5cm}\]
Thus, we see that for \,$\mu_{\,U}\,(\,x \,-\,x_{\,0} \;,\; \delta\,) \;\,>\;\,
1 \,-\, \beta $ , \[ \mu_{\,V}\,\left(\,f(x) \,-\, f(\,x_{\,0}\,) \;,\; t\,\right)
\;=\;\, \mu_{\,V}\,\left(\,f(x) \,-\, f_{\,k}(x) \,+\, f_{\,k}(x) \,-\,
f_{\,k}(\,x_{\,0}\,) \,+\, f_{\,k}(\,x_{\,0}\,) \,-\,
f(\,x_{\,0}\,) \;,\; t\,\right) \] \[\hspace{2.5cm}\geq\;\; \mu_{\,V}\,
\left(\,f(x) \,-\, f_{\,k}(x) \;,\; \frac{t}{3}\,\right) \;\ast\;
\mu_{\,V}\,\left(\,f_{\,k}(x) \,-\, f_{\,k}(\,x_{\,0}\,) \;,\;
\frac{t}{3}\,\right)\] \[\hspace{7.5cm} \;\ast\; \mu_{\,V}\,
\left(\,f_{\,k}(\,x_{\,0}\,)
\,-\, f(\,x_{\,0}\,) \;,\; \frac{t}{3}\,\right)\] \[>\;\;
(\,1 \,-\, r\,) \;\ast\; (\,1 \,-\, r\,) \;\ast\; (\,1 \,-\, r\,)
\;\,=\;\, 1 \,-\, r\] Thus, we have \[ \mu_{\,U}\,(\,x \,-\,x_{\,0}
\;,\; \delta\,) \;\,>\;\,1 \,-\, \beta \;\,\Rightarrow\;\,
\mu_{\,V}\,\left(\,f(x) \,-\, f(\,x_{\,0}\,) \;,\; t\,\right)
\;\,>\;\, 1 \,-\, r\;\;\;\cdots\;\;\;(5)\] Again , for
\,$ \nu_{\,U}\,(\,x \,-\,x_{\,0}
\;,\; \delta\,) \;\,<\;\,\beta $ , \[ \nu_{\,V}\,\left(\,f(x)
\,-\, f(\,x_{\,0}\,) \;,\; t\,\right)
\;=\;\, \nu_{\,V}\,\left(\,f(x) \,-\, f_{\,k}(x) \,+\, f_{\,k}(x) \,-\,
f_{\,k}(\,x_{\,0}\,) \,+\, f_{\,k}(\,x_{\,0}\,) \,-\,
f(\,x_{\,0}\,) \;,\; t\,\right) \] \[\hspace{4.5cm}\leq\;\; \nu_{\,V}\,
\left(\,f(x) \,-\, f_{\,k}(x) \;,\; \frac{t}{3}\,\right) \;\diamond\;
\nu_{\,V}\,\left(\,f_{\,k}(x) \,-\, f_{\,k}(\,x_{\,0}\,) \;,\;
\frac{t}{3}\,\right)\] \[\hspace{7.5cm} \;\diamond\; \nu_{\,V}\,
\left(\,f_{\,k}(\,x_{\,0}\,)
\,-\, f(\,x_{\,0}\,) \;,\; \frac{t}{3}\,\right)\] \[<\;\;
r \;\diamond\; r \;\diamond\; r
\;\,=\;\, r \hspace{1.5cm}\] Hence, we have \[ \nu_{\,U}\,(\,x \,-\,x_{\,0}
\;,\; \delta\,) \;\,<\;\,\beta \;\,\Rightarrow\;\,
\nu_{\,V}\,\left(\,f(x) \,-\, f(\,x_{\,0}\,) \;,\; t\,\right)
\;\,<\;\, r \;\;\;\cdots\;\;\;(6)\] Thus , from $( 5 )$ and $(6)$ it follows
that \,$f$\, is intuitionistic fuzzy continuous on \, \,$(\,U\,,\,A\,)$ .
\end{proof}
\smallskip
\begin{note}
\; The converse of the above theorem is not necessarily true. For example, we consider the sequence of functions of example 4.3. It is obvious that each $\;f_n\;$ is sequentially intuitionistic fuzzy continuous on $\;(0,1)\;$ and hence is intuitionistic fuzzy continuous on $\;(0,1)\;$. Also, the limit function $\;f\;$ is intuitionistic fuzzy continuous on $\;(0,1)\;$, but the intuitionistic fuzzy convergence is not uniformly intuitionistic fuzzy convergent on $\;(0,1)\;.$
\end{note}
\medskip
\section{\textbf{Open Problem :}}
One can develop the concept of differentiation and Riemann integration in an
intuitionistic fuzzy normed linear space and then verify whether the term by term
differentiation and integration are valid or not for a sequence of function in an
intuitionistic fuzzy normed linear space .

\end{document}